\documentclass[12pt]{amsart}

\usepackage[margin = 1in]{geometry}

\usepackage{tikz}
\usepackage{tikz-cd}
\usepackage{url, hyperref}
\usetikzlibrary{shapes.arrows}
\usetikzlibrary{decorations.pathreplacing}

\tikzset{->-/.style={decoration={
  markings,
  mark=at position .45 with {\arrow{>}}},postaction={decorate}}}

\usepackage{amsmath}
\usepackage{amssymb}
\usepackage{enumitem}
\usepackage{graphicx}
\usepackage{mathdots}
\usepackage{color}
\usepackage{diagbox}
\usepackage{array, makecell}
\usepackage{rotating}
\usepackage{amsmath}
\usepackage{tikz-cd}

\def\B{{\boldsymbol B}}

\def\cO{\mathcal{O}}
\def\oM{\overline{\mathcal{M}}}
\def\cM{{\mathcal{M}}}

\def\Z{\mathbb{Z}}
\def\C{\mathbb{C}}
\def\Q{\mathbb{Q}}

\def\b1{{\bf 1}}

\def\E{\mathsf{E}}
\def\L{\mathsf{L}}

\def\H{\mathsf{H}}

\def\D{\mathsf{D}}
\def\P{\mathbb{P}}

\def\cE{\mathcal{E}}
\def\cL{\mathcal{L}}

\usepackage{amsthm}
\newtheorem{definition}{Definition}[section]
\newtheorem{theorem}[definition]{Theorem}

\newtheorem{proposition}[definition]{Proposition}
\newtheorem{corollary}[definition]{Corollary}
\newtheorem{lemma}[definition]{Lemma}
\newtheorem{remark}[definition]{Remark}

\newtheorem{question}[definition]{Question}

\newcommand{\J}{\mathsf{Jac^d}}
\newcommand{\Sym}{\mathsf{Sym}}
\newcommand{\p}{\mathsf{p}}

\newcommand{\Bl}{{\mathsf{Bl}}}
\newcommand{\Tev}{{\mathsf{Tev}}}
\newcommand{\vTev}{{\mathsf{vTev}}}

\newcommand{\init}{\text{init}}

\newcommand{\spine}{{\text{sp}}}

\usepackage{graphicx}

\begin{document}

\title{Fixed-domain curve counts for blow-ups of projective space}

\author{Alessio Cela}
\address{ University of Cambridge, Department of pure mathematics and mathematical statistics
\hfill \newline\texttt{}
 \indent Centre for Mathematical Sciences, Wilberforce Road Cambridge, UK} \email{{\tt ac2758@cam.ac.uk}}

 \author{Carl Lian}
\address{Washington University in St. Louis, Department of Mathematics, 1 Brookings Drive
\hfill \newline\texttt{}
 \indent  St. Louis, MO 63130} \email{{\tt clian@wustl.edu}}

\date{\today}

\maketitle

\begin{abstract}
We study the problem of counting pointed curves of fixed complex structure in blow-ups of projective space at general points. The geometric and virtual (Gromov--Witten) counts are found to agree asymptotically in the Fano (and some $(-K)$-nef) examples, but not in general. For toric blow-ups, geometric counts are expressed in terms of integrals on products of Jacobians and symmetric products of the domain curves, and evaluated explicitly in genus 0 and in the case of $\Bl_q(\P^r)$. Virtual counts for $\Bl_q(\P^r)$ are also computed via the quantum cohomology ring.
\end{abstract}

\tableofcontents

\section{Introduction}

\subsection{Curve-counting with fixed domain}

Let $X$ be a non-singular, irreducible, projective variety over $\C$ of dimension $r$. Curve-counting problems on $X$ are traditionally formulated in Gromov--Witten theory as intersection numbers on the space of stable maps $\oM_{g,n}(X,\beta)$ against the virtual fundamental class $$[\oM_{g,n}(X, \beta)]^{\mathrm{vir}}\in  A_{\mathrm{vdim}(\oM_{g,n}(X, \beta))}(\oM_{g,n}(X, \beta)).$$ For example, fixing $n$ subvarieties $X_i\subset X$, pulling back the classes of the $X_i$ under the evaluation maps $\mathrm{ev}_i:\oM_{g,n}(X, \beta)\to X$, and integrating the product against the virtual class gives a virtual count of the number of genus $g$ curves of class $\beta$ passing through the $X_i$.

We will be concerned with such problems where, in addition, the complex structure of the domain curve $(C,p_1,\ldots,p_n)$ is fixed (and general). Gromov--Witten counts are therefore obtained by additionally restricting to the pullback of a point under the forgetful map $\oM_{g,n}(X,\beta)\to\oM_{g,n}$. 

There are two salient features of the fixed-domain version of the problem: first, degenerating $(C,p_1,\ldots,p_n)$ to a union of rational curves, the answers may be expressed purely in terms of the small quantum cohomology ring of $X$ \cite[Theorem 1.3]{BP}, and therefore tend to be simpler than arbitrary Gromov--Witten invariants. Second, the restriction to a general curve often avoids the most pathological subvarieties of $\oM_{g,n}(X,\beta)$, and therefore, the fixed-domain Gromov--Witten invariants are much more often enumerative \cite{LP, BLLRST}. Furthermore, even when enumerativity fails, the corresponding geometric count, which is different, is usually well-defined and may still be accessible via more subtle methods.

In this paper, we will restrict to the case in which the subvarieties $X_i\subset X$ are points $x_i$, but many of our methods work more generally. We now set up the problem precisely.

Fix integers $g,n \geq 0$ satisfying the stability condition $2g-2+n >0$ so that the moduli space $\oM_{g,n}$ of $n$-pointed, genus $g$ stable curves is well-defined. Fix $\beta \in H_2(X, \Z)$ an effective curve class satisfying the condition
$$
\beta\cdot K_X^\vee>0.
$$

Let $\oM_{g,n}(X,\beta)$ be the moduli space of genus $g$, $n$-pointed stable maps to $X$ in class $\beta$ and assume that the dimensional constraint
$$
\mathrm{vdim}(\oM_{g,n}(X,\beta))=\mathrm{dim}(\oM_{g,n} \times X^n)
$$
holds. This is equivalent to
\begin{equation}\label{dim constraint}
\beta\cdot K_X^\vee=r(n+g-1).
\end{equation}

If the dimension constraint \eqref{dim constraint} holds, we expect a finite number of maps from a fixed $n$-pointed curve $(C,p_1,\ldots,p_n)$ of genus $g$ to $X$ in class $\beta$ where the $p_i$ are incident to general points of $X$. Unless otherwise specified, we will always assume that constraint \eqref{dim constraint} holds throughout.

The \textbf{virtual Tevelev degree} $\mathsf{vTev}^X_{g,n,\beta}$ of $X$ is defined in \cite[Definition 1.1]{BP} to be the corresponding virtual count. Formally, let
\begin{equation}\label{map tau}
\tau: \oM_{g,n}(X,\beta) \rightarrow \oM_{g,n} \times X^{n}
\end{equation}
be the canonical morphism obtained from the domain curve and the evaluation maps.
Then, $\mathsf{vTev}^X_{g,n,\beta} \in \Q$ is defined by the equality
$$
\tau_*[\oM_{g,n}(X, \beta)]^{\mathrm{vir}}= \mathsf{vTev}^X_{g,n,\beta}\cdot [\oM_{g,n} \times X^n] \in A^0(\oM_{g,n} \times X^n)_{\Q}.
$$

The \textbf{geometric Tevelev degree} $\mathsf{Tev}^X_{g,n,\beta} \in \Z$ of $X$ is defined in \cite[Definition 2]{LP} under the further assumption that the restriction of $\tau$ to maps with smooth domain
$$
\tau: \cM_{g,n}(X,\beta) \rightarrow \cM_{g,n} \times X^{ n}
$$
has reduced and $0$-dimensional general fiber, in which case its cardinality is by definition equal to $\mathsf{Tev}^X_{g,n,\beta}$. The integrality of $\mathsf{Tev}^X_{g,n,\beta}$ follows from the absence of automorphisms in a general fiber of $\tau$, see \cite[Remark 4]{LP}.

We record here the following crucial lemma, see \cite[Proposition 14, proof of Proposition 22]{LP}. It does not require the assumption \eqref{dim constraint}.

\begin{lemma}\label{basic_transversality}
Let $Z\subset \cM_{g,n}(X,\beta)$ be an irreducible component. Suppose that $Z$ dominates $\cM_{g,n}\times X^n$ and that $n\ge g+1$. Then, $Z$ is generically smooth of dimension equal to the virtual dimension.

More generally, let $Z\subset \oM_{g,n}(X,\beta)$ be an irreducible component whose general point $f:C\to X$ has the property that $C$ is a union of a smooth genus $g$ curve $C_{\spine}$ containing $m$ of the marked points, and $n-m$ rational tails attached to $C_{\spine}$, each containing a marked point. Suppose that $Z$ dominates $\cM_{g,n}\times X^n$ and that $m\ge g+1$. Then, $Z$ is generically smooth of dimension equal to the virtual dimension.
\end{lemma}
In particular, $\mathsf{Tev}^X_{g,n,\beta}$ is well-defined under assumption \eqref{dim constraint} whenever $n\ge g+1$. This bound is not sharp: for example, when $X=\P^r$, the geometric Tevelev degrees are also well-defined whenever $n\ge r+1$ (in which case $f:C\to\P^r$ is necessarily non-degenerate) by Brill--Noether theory. It is an interesting question to determine for other $X$ whether the bound $n\ge g+1$ can be improved.

Three questions emerge.
\begin{enumerate}
    \item What is $\mathsf{vTev}^X_{g,n,\beta}$?
    \item What is $\mathsf{Tev}^X_{g,n,\beta}$?
    \item Is $\mathsf{vTev}^X_{g,n,\beta}=\mathsf{Tev}^X_{g,n,\beta}$? (That is, is $\mathsf{vTev}^X_{g,n,\beta}$ enumerative?)
\end{enumerate}

The three questions are clearly not independent, but in practice involve somewhat different ideas interacting in subtle ways.

This paper is concerned with aspects of all three questions when $X$ is a blow-up of $\P^r$ at a collection of points. Other examples of interest include complete intersections, homogeneous spaces and Hirzebruch surfaces \cite{LP,BP,Cela,Lian,CIL}. We next review the preliminary case $X=\P^r$.

\subsection{Curve counts on $\P^r$}

The answers to all three questions above on fixed-domain curve-counts are fully understood for $X=\P^r$.

Virtual counts for $\P^r$ are easy to obtain. We write $\beta=d$ for the class of $d$ times a line throughout this section.

\begin{theorem}\cite[equation (3)]{BP}
Assume \eqref{dim constraint}. Independently of $d$, we have
\begin{equation*}
\vTev^{\P^r}_{g,n,d}=(r+1)^g.
\end{equation*}
\end{theorem}
In fact, this formula had been obtained much earlier \cite{BDW} before the modern theory of the Gromov--Witten virtual fundamental class was available. It is also a consequence of the Vafa--Intriligator formula \cite{ST,MO}.

The virtual and geometric counts agree when $d$ is large compared to $r$ and $g$.
\begin{theorem}\cite{FL}\label{geomtev_larged}
Assume \eqref{dim constraint} and that $d\ge rg+r$ (equivalently, $n\ge d+2$). Then,
\begin{equation*}
\Tev^{\P^r}_{g,n,d}=(r+1)^g.
\end{equation*}
\end{theorem}

Obtaining geometric counts in \emph{low} degree for $\P^r$ is much more difficult; the virtual and geometric counts no longer agree. For $\P^1$, we have the following.
\begin{theorem}\cite{CPS,FL,CL}\label{tev_p1}
Assume \eqref{dim constraint} for $r=1$. Then, 
\begin{align*}
\Tev^{\P^1}_{g,n,d}&=2^g-2 \sum_{j=0}^{g-d-1}\binom{g}{j}+(g-d-1)\binom{g}{g-d}+(d-g-1)\binom{g}{g-d+1}\\
&=\int_{\mathrm{Gr}(2,d+1)}\sigma^g_{1}\cdot\sum_{a_0+a_1=n-3}\sigma_{a_0}\sigma_{a_1}
\end{align*}
\end{theorem}
Binomial coefficients $\binom{g}{j}$ with $j<0$ are interpreted to vanish, recovering Theorem \ref{geomtev_larged} when $d\ge r+1$. See \cite[\S 2]{FL} for notation for the Schubert cycles appearing in the last formula.

Finally, the computation of geometric counts $\Tev^{\P^r}_{g,n,d}$ is completed in \cite{lian_pr}, given in terms of Schubert calculus. We do not restate the formula here.

\subsection{New results}

In this paper, we study Tevelev degrees of blow-ups of projective spaces at (general) points. We obtain new results on all three aspects of the question: enumerativity, geometric, and virtual calculations.

\subsubsection{Enumerativity}

Fixed-domain virtual curve counts are much more likely to be enumerative than arbitrary Gromov--Witten invariants. For example, we have already seen that virtual and geometric counts for $\P^r$ agree in large degree, whereas higher-genus Gromov--Witten invariants of $\P^r$ typically fail to be enumerative. This phenomenon was first studied systematically in \cite{LP}.


\begin{definition}\label{def:SAE}
We say that $X$ satisfies \textbf{strong asymptotic enumerativity (SAE)} if :
    \begin{enumerate}
        \item For all $g \geq 0$, there exists a constant $C(X,g)$  (depending only on $X$ and $g$, but not on $\beta$) for which, if $n\ge C(X,g)$ and the dimensional constraint \eqref{dim constraint} is satisfied, then the general fiber of the map \eqref{map tau} $$\tau: \oM_{g,n}(X,\beta) \rightarrow \oM_{g,n} \times X^{n}$$ is contained in $\cM_{g,n}(X,\beta)$, and
        \item one can take $C(X,0)=3$.
    \end{enumerate}
\end{definition}

Recall that we require $2g-2+n>0$, so when $g=0$, it must be the case that $n\ge 3$. Thus, (2) above asserts that when $g=0$, the condition on the general fiber holds for all values of $n$ that we consider. If $\tau$ has general fiber contained in $\cM_{g,n}(X,\beta)$ and furthermore $n \geq g+1$, then by Lemma \ref{basic_transversality}, the general fiber of $\tau$ is reduced and $0$-dimensional, so $\mathsf{Tev}^X_{g,n,\beta}=\mathsf{vTev}^X_{g,n,\beta}$, i.e. $\mathsf{vTev}^X_{g,n,\beta}$ is enumerative. Thus, we have the following.
\begin{proposition}\label{SAE_implies_enum}
Suppose that $X$ satisfies SAE. Then, virtual Tevelev degrees are enumerative:
\begin{itemize}
\item in arbitrary genus whenever $n$ (equivalently, $\beta\cdot K_X^\vee$) is sufficiently large (depending on $g$), and 
\item \emph{always} when $g=0$.
\end{itemize}
\end{proposition}

In \cite{LP}, SAE is proven for homogeneous spaces for linear algebraic groups (see \cite[Theorem 10]{LP}) and non-singular hypersurfaces of very low degree (see \cite[Theorem 11, Corollary 34]{LP}), implying the conclusion of Proposition \ref{SAE_implies_enum} in these examples. As it is difficult to imagine how one could obtain this conclusion for $X$ without SAE, the following question is natural.
 
 \begin{question}
What geometric conditions on (non-singular, projective) $X$ guarantee SAE?
 \end{question}

It was originally expected that all \emph{Fano} $X$ satisfy SAE \cite[Speculation 12 and \S 4]{LP}. However, several families of counterexamples to this speculation have been constructed in \cite[\S 3]{BLLRST}. A symplectic analog of \cite[Speculation 12]{LP} has more recently been established in \cite{CD}.

Our first main set of results concerns SAE for $X$ given by the blow-up of $\P^r$ at distinct points $q_1,\ldots,q_\ell$. We have the following negative result.

\begin{theorem}\label{enum_thm_negative}
Suppose that $r\ge4$ and $\ell \ge2$. Then, $X$ fails to satisfy SAE.
\end{theorem}

Such $X$ are not Fano, and indeed, the existence of negative rational curves is used in an in essential way. In fact, we establish a more general criterion for failure of SAE which also applies to various special configurations of points $q_j$ (e.g., where some three of the $q_j$ are collinear), see Proposition \ref{how to disprove}. 

On the other hand, we have the following positive result.

\begin{theorem}\label{enum_thm_positive}
$X=\Bl_{q_1,\ldots,q_\ell}(\P^r)$ satisfies SAE in the following cases:
\begin{enumerate}
    \item (Theorem \ref{thm: SAE for Del Pezzo surfaces}) $X$ is del Pezzo, i.e. $r=2$, and the $\ell\le 8$ points satisfy the property that no three lie on a line, no six lie on a conic, and, if $\ell=8$, the points do not all lie on a cubic singular at one of the $q_j$.
    \item (Theorem \ref{SAE for blow-up of P3 at 4 pts}) $r=3,\ell\le4$, and the $q_j$ are (linearly) general.
    \item (Theorem \ref{SAE for blow-ups at 1 point}) $r$ is arbitrary and $\ell=1$.
\end{enumerate}
\end{theorem}

 In Theorem \ref{enum_thm_positive}, examples 1 and 3 are Fano; in fact, these are all examples of Fano varieties obtained by blowing up $\P^r$ at general points. In example 2, $X$ is merely $(-K_X)$-nef when $\ell\ge2$; these are the first examples of non-Fano varieties proven to satisfy SAE.

When the points $q_j$ are general, the only cases left open are when $r=2,\ell\ge9$ and $r=3,\ell\ge5$. We believe SAE should hold in these examples, but we do not see a way to obtain unconditional proofs. In dimension 2, 
one would at least need a bounded negativity statement for general blow-ups of $\P^2$, if not the full power of the SHGH conjecture (see \cite{H} for a survey), and the situation in dimension 3 seems even more hopeless.
    
\subsubsection{Geometric calculations}\label{geom_intro}

In \S\ref{section: proof of integral formula}, we prove an integral formula for the geometric Tevelev degrees of blow-ups $\pi:X \rightarrow \P^r$ at $\ell \leq r+1$ general points $q_1,\ldots,q_{\ell}$. Without loss of generality, we will take the $q_j$ to be fixed points under the standard action of a $(r+1)$-dimensional torus on $\P^r$.

Let $\H\in H^2(X)$ denote the pullback of the hyperplane class from $\P^r$, and let $\E_1,\ldots,\E_\ell\in H^2(X)$ be the classes of the exceptional divisors. Let $\H^\vee,\E^\vee_1,\ldots,\E^\vee_\ell\in H_2(X)$ be the corresponding dual basis, and let
$$
\beta= d \H^\vee + \sum_{i=1}^\ell k_i \E_i^\vee\in H_2(X,\Z)
$$
be an effective curve class.

We are interested in counting maps from a fixed general curve $(C,p_1,\ldots,p_n)$ to $X$ in class $\beta$. Equivalently, we count maps to $\P^r$ of degree $d$ and with multiplicities $k_1,\ldots,k_{\ell}$ at the points $q_1,\ldots,q_\ell$, respectively. Such a map is defined by $r+1$ sections of a line bundle $\cL$ of degree $d$, along with $\ell$ divisors $D_1,\ldots,D_\ell\subset C$ of degrees $k_1,\ldots,k_\ell$, respectively, which are required to map to $q_1,\ldots,q_\ell$, respectively. When $\ell\le r+1$, this is equivalent to the data of $r+1$ sections
$$f_i\in H^0\left(C,\cL\left(- \sum_{j \neq i} D_j\right)\right)$$
which have no common vanishing point \emph{viewed as sections of $H^0(C,\cL)$.} We will parametrize such data in what follows.

\vspace{5pt}
\noindent  \textbf{Setup}
Write $S:=\J(C) \times \Sym^{k_1}(C) \times \ldots  \times \Sym^{k_\ell}(C)$. Over $S\times C$, we have the following.
\begin{enumerate}
\item[$\bullet$] the pullback $\mathcal{P}$ from $\J(C) \times C$ of a Poincar\'{e} line bundle;
\item[$\bullet$] for all $i=1,\ldots.,\ell$, the pullback $\mathcal{D}_i$ from $\Sym^{k_i}(C) \times C$ of the universal divisor.
\end{enumerate}

Let $\nu : S \times C \rightarrow S$
be the projection and define
$$
\cE:=  \bigoplus_{i=1}^{r+1} \nu_* \left( \mathcal{P}\left(- \sum_{j \neq i} \mathcal{D}_j\right) \right).
$$
We will assume that
\begin{equation}\label{stronger inequality d, ks}
d - \sum_{i \in I} k_i > 2g-1 \text{ for all } I \subseteq \{1,\ldots,\ell \} \text{ such that  } |I| \leq r,
\end{equation}
in order to ensure that $\cE$ is a vector bundle. We will see later in \S\ref{section: proof of integral formula} that we will also want this to remain true upon replacing $\cL$ with $\cL(-p_i)$, hence the appearance of $2g-1$ rather than $2g-2$ on the right hand side.

Let $\eta: \P:=\P(\cE) \rightarrow S$
be the corresponding projective bundle corresponding to \emph{lines} (not 1-dimensional quotients) in the fibers of $\cE$. Therefore, a point of $\P$ consists of the data of a point of $S$ and sections $f_j$ as above, taken up to simultaneous scaling.

Our first result expresses Tevelev degrees of $X$ as integrals on $\P$. Before stating the result, we require some additional notation. We set
\begin{equation}\label{Definition tilde(H)}
\widetilde{\H}= c_1( \mathcal{O}_\P(1)) -\eta_1-\cdots-\eta_\ell \in H^*(\P,\Z)
\end{equation}
where $\eta_i$ is the pullback from $\Sym^{k_i}(C)$ of the class of the divisor $N_i=\{D: D-p \geq 0 \} $ (here $p$ is any fixed point in $C$). 

Then, we have the following.

\begin{proposition}[Integral formula]\label{thm:integral formula}
    Assume conditions \eqref{dim constraint}, \eqref{stronger inequality d, ks} hold, and in addition that 
\begin{equation}\label{eqn: range for d to have enumerativity}
n-d\ge g+1.
\end{equation}

Then, $\mathsf{Tev}_{g,n,\beta}$ is well-defined and the following integral formula holds:
    \begin{equation}\label{eqn:integral formula}
        \mathsf{Tev}_{g,n,\beta}^X= \int_{\P} (\widetilde{\H}^r+\sigma_1 \widetilde{\H}^{r-1}+\cdots+ \sigma_{r})^n,
    \end{equation}
    where $\sigma_i$ is the $i$-th elementary symmetric function in $\eta_1,\ldots,\eta_\ell$ (so $\sigma_i=0$ for $i>\ell$).
\end{proposition}

The condition $n-d\ge g+1$ is needed to rule out extraneous contributions when the sections $f_j$ are identically zero (see Lemma \ref{coordinate_nonzero} and Remark \ref{ineq_necessary}), and also ensure dimensional transversality later (Lemma \ref{actual_dim_after_twisting}). The factor $\widetilde{\H}^r+\sigma_1 \widetilde{\H}^{r-1}+\cdots+ \sigma_{r}$ is the class of a subscheme $V(x_i)$ of $\P$ corresponding to the condition that $f(p_i)=x_i$. However, care is needed to deal with the possibility that the sections $f_j$ vanish simultaneously at $p_i$; a precise definition of $V(x_i)$ is given in \S\ref{section: proof of integral formula}. The computation of the class $[V(x_i)]$ and the fact that the $V(x_i)$ intersect transversely at points enumerated by $\mathsf{Tev}_{g,n,\beta}^X$ are both subtle.

Pushing forward to $S=\J(C) \times \Sym^{k_1}(C) \times \cdots  \times \Sym^{k_\ell}(C)$ and using Grothendieck-Riemann-Roch, we obtain (see \S\ref{coh_of_S} for notation) the following.

\begin{theorem}\label{thm:grr}
Assume conditions \eqref{dim constraint}, \eqref{stronger inequality d, ks} and \eqref{eqn: range for d to have enumerativity} hold. Then,
 $$
 \mathsf{Tev}_{g,n,\beta}^X=\sum_{m=0}^{\min(n,k_1,\ldots,k_{r+1})} \binom{n}{m}(-1)^m \int_{S} \prod_{i=1}^{r+1}(1+\eta_i)^{n-m-1+g-d+\overline{k}_i} \eta_i^m  \cdot \mathrm{exp}\left( \frac{ \overline{\tau}_i+ \Theta- \overline{x}_i}{1+\eta_i} \right).
$$
\end{theorem}

We extract from Theorem \ref{thm:grr} relatively simple formulas in two special cases.

\begin{theorem}[$g=0$ specialization]\label{explicit formulas in genus 0}
    Let $X$ be the blow-up of $\P^r$ at $\ell \leq r+1$ points and let 
    $$
    \beta=d\H^\vee+ \sum_{i=1}^\ell k_i \E_i^\vee \in H_2(X,\Z)
    $$
    be an effective curve class. Assume conditions \eqref{dim constraint}, \eqref{stronger inequality d, ks}, and \eqref{eqn: range for d to have enumerativity}. Then,
    $$
        \mathsf{Tev}_{0,n,\beta}^X= \sum_{m=0}^{\min(k_1,\ldots,k_{r+1},n)} (-1)^m \binom{n}{m} \prod_{i=1}^{r+1}\binom{n-d+\sum_{j \neq i}k_j-1-m}{k_i-m}
    $$
    where we set $k_{\ell+1}=\cdots=k_{r+1}=0$ when $\ell < r+1$.
\end{theorem}
In particular, when $\ell<r+1$, the summation goes away, and $\mathsf{Tev}_{0,n,\beta}^X$ is a product of binomial coefficients.

It is often the case that the formula above gives zero. For example, take $\ell=1$ and $\beta=d \H^\vee+k \E^\vee$ (where we write simply $\E_1^\vee=\E^\vee$), where $(r-1)k\le d$. Assuming \eqref{dim constraint}, the inequalities \eqref{stronger inequality d, ks}, and \eqref{eqn: range for d to have enumerativity} are satisfied. Then, Theorem \ref{explicit formulas in genus 0} reads simply
$$
\mathsf{Tev}_{0,n,\beta}^X=\binom{n-d-1}{k}.
$$
If we assume further that $(r-1)k\le d< (2r-1)k$, then we have $0\le n-d-1< k$, so we conclude that $\mathsf{Tev}_{0,n,\beta}^X=0$. In other words, the map 
$$\tau: \cM_{g,n}(X,\beta) \rightarrow \cM_{g,n} \times X^{n}$$
fails to be dominant. This is explained by the instability of the tangent bundle $T_X$ under the slope function given by
\begin{equation*}
    \mu_\beta(\cE)=\frac{\int_X c_1(\cE)\cdot\beta}{\text{rank}(\cE)}.
\end{equation*}
If $X$ is viewed as a $\P^1$ bundle over $\P^1$, then the relative tangent bundle is $\mu_\beta$-destabilizing. We thank the other authors of \cite{BLLRST} for providing this explanation.

Specializing Theorem \ref{thm:grr} to $\ell=1$ in arbitrary genus yields the following.

\begin{theorem}[$\ell=1$ specialization]\label{explicit formulas for 1 pt blow-up}
Let $X=\Bl_q(\P^r)$ and let 
    $$
    \beta=d\H^\vee+  k \E^\vee \in H_2(X,\Z)
    $$
    be an effective curve class. Assume that conditions \eqref{dim constraint}, \eqref{stronger inequality d, ks}, and \eqref{eqn: range for d to have enumerativity} hold. Then,
    $$
        \mathsf{Tev}_{g,n,\beta}^X= \sum_{m=0}^g (2r)^{g-m} (1-r)^m \binom{g}{m} \binom{n-d+g-m-1}{k}.
    $$
\end{theorem}

Obtaining explicit formulas beyond these two specializations is cumbersome in general. For example, we have checked that, if $X$ is the blow-up of $\P^2$ at two points and $\beta\cdot K_X^\vee$ is sufficiently large, then $\Tev^X_{g,n,\beta}$ is equal to
\begin{align*}
\sum_{\substack{a_1+b_1+b_2+b_3+a_3+a_4=g \\ b_1+b_2+b_3=a_2}} \binom{g}{a_1 \ a_2 \ a_3 \ a_4} \binom{a_2}{b_1 \ b_2 \ b_3}5^{a_1} (-1)^{b_1+a_3+a_4} \sum_{\ell=0}^{\mathrm{min}(k_1+a_3-a_2,k_2+a_4-a_2,b_1)} (-1)^\ell \binom{b_1}{\ell} \cdot \\ \binom{2d-(k_1-a_3)-g-n+1-b_2-b_3- \ell}{k_1-a_2-b_2-\ell} \cdot \binom{2d-(k_2-a_4)-g-n+1-b_2-b_3- \ell}{k_2-a_2-b_3-\ell}.
\end{align*}

Note that all of our geometric results require the inequalities \eqref{stronger inequality d, ks} and \eqref{eqn: range for d to have enumerativity}. If $g$ and $k_j$ are viewed as fixed, then both inequalities hold for sufficiently large $d$, so our results may be viewed as ``large degree'' counts in analogy with Theorem \ref{geomtev_larged}. The low degree counts seem, unsurprisingly, more difficult, and are left open.

We remark that when $d - (k_1+\cdots+k_r) < 0$ and $n\ge1$, a map $f:C\to X$ as above would need to have image contained in the strict transform of the torus-invariant hyperplane where the last coordinate is zero. On the other hand, as the $x_i$ are general, they may be chosen not to lie on this hyperplane, so it follows immediately that $\mathsf{Tev}_{g,n,\beta}^X=0$ in this case.

\subsubsection{Virtual calculations}

Via a computation in the quantum cohomology ring $QH^{*}(X)$, we also obtain the following.

\begin{theorem}\label{virtual formulas for 1 pt blow-up}
    Let $X=\Bl_q(\P^r)$ and let 
    $$
    \beta=d\H^\vee+  k \E^\vee \in H_2(X,\Z)
    $$
    be an effective curve class. Assume that \eqref{dim constraint} holds and that $n-d \geq 1$. Then,
    $$
        \mathsf{vTev}_{g,n,\beta}^X= \sum_{m=0}^g (2r)^{g-m} (1-r)^m \binom{g}{m} \binom{n-d+g-m-1}{k}.
    $$
\end{theorem}
Thus, the formula matches that of Theorem \ref{explicit formulas for 1 pt blow-up}, but holds in a wider range.

\subsection{Acknowledgments}

Portions of this work were completed during the first author's visits to HU Berlin in November 2022 (during the workshop ``Resonance varieties, topological invariants of groups, moduli'') and February 2023 (during the Northern German Algebraic Geometry Seminar) and the second author's visits to ETH Z\"{u}rich in July 2022 (during the ICM satellite meeting, Algebraic Geometry and Number Theory) and November 2022. We thank both institutions for their continued support. We also thank Gavril Farkas, Zhuang He, Woonam Lim, Rahul Pandharipande, Eric Riedl, and Johannes Schmitt for helpful discussions, and the referee for many useful corrections and suggestions.

A.C. was supported by SNF-200020-182181 and SNF-222363. C.L. was supported by an NSF postdoctoral fellowship, grant DMS-2001976, and the MATH+ incubator grant ``Tevelev degrees.'' 

\section{Enumerativity}

\subsection{Failure of SAE}\label{section:How to disprove SAE}
We give a general construction of $X$ failing SAE (Definition \ref{def:SAE}).

\begin{proposition}\label{how to disprove}
    Let $\pi: X \rightarrow \P^r$ be the blow-up of $\P^r$ at any set of points $q_1,\ldots,q_m\in\P^r$. Suppose that there exists a curve $\P^1\cong D\subset X$, such that either:
    \begin{enumerate}
        \item $D\cdot K_X^{\vee}<0$, or
        \item $D\cdot K_X^\vee=0$, $r=2$, and $\pi(D)$ is not a point.
    \end{enumerate}
    Then, $X$ fails to satisfy SAE.
\end{proposition}
\begin{proof}
    We will prove that for every $g$, there is a divergent sequence $\{n_k=n_k[g]\}_k$ and a corresponding sequence of curve classes $\beta_k\in H_2(X)$ satisfying \eqref{dim constraint}, such that the general fiber of the map $\tau$ contains maps with singular domain for all $n_k$.
    
    We start by assuming that $D$ has negative degree against $K_X^{\vee}$. Fix $g \geq 0$, and for $d \geq 0$, set 
    $$
    \beta[d]=d \H^{\vee}+ r D \in H_2(X,\Z),
    $$
    where $\H^{\vee}$ is the pullback of the class of a line on $\P^r$. Notice that $\beta[d] \cdot K_X^\vee$ diverges for large $d$. Let $n$ be such that equation \eqref{dim constraint} is satisfied with $\beta=\beta[d]$, i.e. 
    $$
    d(r+1)+r (D\cdot K_X^\vee)=r(n+g-1);
    $$
    in order for this to be possible, we assume that $d$ is divisible by $r$.
    
    Let 
    $$
    \cM_{\Gamma}\subseteq \partial \oM_{g,n}(X,\beta)
    $$
    be the locally closed locus consisting of stable maps where the domain curve consists of a genus $g$ smooth curve mapping to $X$ with class $d \H^\vee$ and containing all $n$ of the markings, attached to a smooth rational tail mapping to $X$ with class $rD$. We claim that $\cM_{\Gamma}$ dominates $\oM_{g,n} \times X^n$, so that SAE fails.
    
    To see this, notice that $m=n-D\cdot K_X^\vee >n $ satisfies the dimensional constraint \eqref{dim constraint} for $\P^r$ with curve class $d\H^\vee$. Therefore, since $\mathsf{Tev}^{\P^r}_{g,m,d\H^\vee} >0$ (for example, we have $m\ge d+2$ for large $d$, so we can apply \cite[Theorem 1.1]{FL}), for general $(C,p_1,\ldots,p_m) \in \cM_{g,m}$ and $x_1,\ldots,x_m \in \P^r$ there exists 
    $$
    f: (C,p_1,\ldots,p_m) \rightarrow \P^r
    $$
    in class $d\H^\vee$ mapping $p_i$ to $x_i$ for all $i=1,\ldots,m$.

    We are also free to assume that $x_{n+1}\in \pi(D)$. From $f$, we construct a new map
    $$
    \tilde{f}:(\tilde{C},p_1,\ldots,p_n)  \rightarrow X
    $$
    where $\widetilde{C}$ is obtained from $C$ by attaching a $\P^1$ at $p_{n+1}$ and forgetting $p_i$ for $i>n+1$. The restriction $\widetilde{f}$ to $C$ is the unique morphism such that $\pi \circ \widetilde{f}=f$, and the restriction to $\P^1$ is a fixed degree $r$ cover of $D$. The map $[\widetilde{f}]$ lies in $\cM_\Gamma$, and maps to a general point of $\oM_{g,n} \times X^n$, as needed.
    
    The proof in the second case is analogous: we have $m=n$ above, so we cannot add an $(n+1)$-st point on $C$ constrained to map to $D$, but note that $f(C) \cap \pi(D) \neq \emptyset$ automatically, so we can construct a map $\tilde{f}$ as before.
\end{proof}

\begin{proof}[Proof of Theorem \ref{enum_thm_negative}]
Apply Proposition \ref{how to disprove} with $D$ equal to the strict transform of the line between any two of the blown up points.
\end{proof}

Similarly, any blow-up of $\P^r$ at a set containing three collinear points fails to satisfy SAE.

\subsection{Generalities for Fano varieties}\label{Section: How to prove SAE for Fano varieties}
We now collect general statements used to prove SAE for Fano varieties; in this subsection, we will always assume that $X$ is Fano. Let $r$ be the dimension of $X$.

\begin{definition}
    Let $ \beta \in H_2(X,\Z)$ be an effective curve class (possibly 0).  We say that $\beta$ is \textbf{ordinary} if the evaluation map $\cM_{0,1}(X,\beta) \rightarrow X$ is dominant.
\end{definition}

Fix a genus $g \geq 0$, a number $n$ of markings satisfying $2g-2+n>0$ and a non-zero effective curve class $ \beta$ on $X$. 

For integers $a\ge0$ and $m\in[0,n]$, define 
$$
\cM_\Gamma^{(m,a)} \subseteq  \oM_{g,n}(X,\beta)
$$
to be the locally closed locus parametrizing maps whose domain $C$ has the following topological type. $C$ contains a smooth \emph{spine} component $C_{\spine}$ of genus $g$, containing the marked points $p_{n-m+1},\ldots,p_n$. Attached to $C_{\spine}$, we have trees of rational curves $T'_1,\ldots,T'_a$, as well as trees of rational curves $T_1,\ldots,T_{n-m}$, such that the tree $T_{i}$ contains the marked point $p_i$ for $i=1,\ldots,n-m$. See also Figure \ref{singular_stable_map}. 
\begin{figure}
	\begin{center}
		\begin{tikzpicture} [xscale=0.36,yscale=0.36]
                \draw [ultra thick, black] (0,7) to (19,7);
                \node at (15,7) {$\bullet$};
                \node at (15,8) {$p_{n-m+1}$};
                \node at (18,7) {$\bullet$};
                \node at (18,8) {$p_n$};
                \node at (19,6) {$C_{\spine}$};
                 \draw [thick, dotted] (16,6) -- (17,6);
                \draw [ultra thick, black] (1,8) to (1,3);
                \draw [ultra thick, black] (0,5) to (2,2);
                \draw [ultra thick, black] (2,3) to (0,0);
                \node at (0,-1) {$T_1'$};
                \draw [ultra thick, dotted] (3,4) -- (4.5,4);
                \draw [ultra thick, black] (5,8) to (6,3);
                \draw [ultra thick, black] (6,5) to (5,1);
                \node at (5,0) {$T_a'$};
                \draw [ultra thick, black] (8,8) to (8,3);
                \node at (8,4) {$\bullet$};
                \node at (7.3,4) {$p_1$};
                \node at (8,2) {$T_1$};
                \draw [ultra thick, dotted] (9,4) -- (10.5,4);
                \draw [ultra thick, black] (13,8) to (11,3);
                \draw [ultra thick, black] (11,6) to (14,1);
                \draw [ultra thick, black] (14,3) to (12,0);
                \draw [ultra thick, black] (11.6,6) to (16,4);
                \node at (13,2.65) {$\bullet$};
                \node at (11.6,2.55) {$p_{n-m}$};
                \node at (12,-1) {$T_{n-m}$};
		\end{tikzpicture}
	\end{center}
	\caption{Topological type of the domain curves of points in $\cM_\Gamma^{(m,a)}$\label{figure-1}\label{singular_stable_map} }
\end{figure}
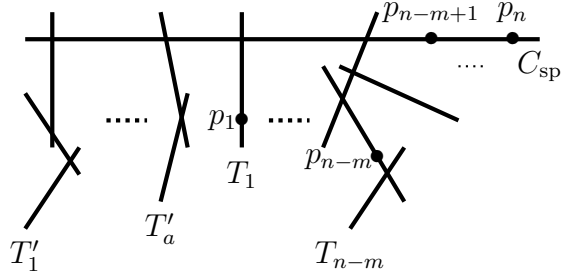

Up to permuting indices of the marked points, the $\cM_\Gamma^{(m,a)}$ contain all of the boundary strata that dominate $\oM_{g,n}$. Recall that we are interested in the question of whether the general fiber of the map $$\tau: \oM_{g,n}(X,\beta) \rightarrow \oM_{g,n} \times X^{n}$$ is contained in $\cM_{g,n}(X,\beta)$. This amounts to the question of whether $\cM_\Gamma^{(m,a)}$ fails to dominate $\oM_{g,n} \times X^n$ for all $(m,a)\neq (n,0)$. 





Let $Z$ be an irreducible component of $\cM_\Gamma^{(m,a)}$ and let $[f] \in Z$ be a general point. Consider the following two conditions on $f$.

\begin{enumerate}
    \item[(*)]\label{condition 1} $a=0$, and the pushforward under $f$ of each component of $T_i$ is an ordinary class for $i=1,\ldots,n-m$. Moreover, we have the inequality
    \begin{equation}\label{eq:dim-vdim}
    \mathrm{dim}_{[f|_{C_{\spine}}]}(\oM_{g,n}(X,f_*[C_{\spine}])) - \mathrm{vdim}(\oM_{g,n}(X,f_*[C_{\spine}]))\leq n-(g+1).
    \end{equation}
    \item[(**)]\label{condition 2} For all $i=1,\ldots,n-m$, either $f_*[T_i] \cdot K_X^\vee \geq r+1$ or the pushforward under $f$ of each component of $T_i$ is an ordinary class. 
\end{enumerate}


\begin{proposition} \label{prop1: how to prove SAE}
Suppose that $g,n,\beta$ satisfy
 \begin{equation} \label{inequality: dim constraint}
 \mathrm{vdim} (\oM_{g,n}(X,\beta)) \leq \mathrm{dim}(\oM_{g,n} \times X^n),
 \end{equation}
or equivalently that $\beta\cdot K_X^\vee\le r(n+g-1)$. Take $a=0$, and let $f\in Z\subset \cM_\Gamma^{(m,0)}$ be as above. Suppose further that if $m=n$ (i.e., $C$ is irreducible), then \eqref{inequality: dim constraint} is a \emph{strict} inequality.
    
Assume that condition (*) holds. Then, $Z$ fails to dominate $\oM_{g,n} \times X^n$.
\end{proposition}
\begin{proof}
Suppose first that $m<g+1$. We estimate $\mathrm{dim}_{[f]}(Z)$ as in \cite[Proof of Proposition 22]{LP}. First, the spine curve $C_{\spine}$ of $f$ moves in a family of dimension at most $n-(g+1)$ more than the expected, by \eqref{eq:dim-vdim}. The remaining components of $C$ are ordinary rational components, which can be attached to $C$ one at a time. Passing through the point of attachment imposes the expected number of conditions on each rational component, so it follows that 
\begin{equation*}
\mathrm{dim}_{[f]}(Z)-\mathrm{vdim}(\cM_\Gamma^{(m,a)}) \le n-(g+1)
\end{equation*}
as well. On the other hand, we have 
\begin{equation*}
\mathrm{vdim}(\cM_\Gamma^{(m,a)})\le \mathrm{vdim} (\oM_{g,n}(X,\beta))-(n-m)\le \mathrm{dim}(\oM_{g,n} \times X^n)-(n-m),
\end{equation*}
by \eqref{inequality: dim constraint}. Therefore,
\begin{equation*}
        \mathrm{dim}_{[f]}(Z) \leq \mathrm{dim}(\oM_{g,n} \times X^n) +m-(g+1)<\mathrm{dim}(\oM_{g,n} \times X^n),
    \end{equation*}
so $Z$ cannot dominate $\oM_{g,n} \times X^n$.

Suppose instead that $m\ge g+1$, and that $Z$ dominates $\oM_{g,n} \times X^n$. Then, by \cite[Proposition 13]{LP}, we have 
    $$
    \mathrm{dim}_{[f|_{C_{\spine}}]}(\oM_{g,n}(X,f_*[C_{\spine}])) = \mathrm{vdim}(\oM_{g,n}(X,f_*[C_{\spine}])).
    $$
Arguing as in the previous case, it follows after attaching rational components that $\mathrm{dim}(Z) = \mathrm{vdim}(\cM_\Gamma^{(m,a)})$. On the other hand, we also have $\mathrm{vdim}(\cM_\Gamma^{(m,a)})\le \mathrm{vdim} (\oM_{g,n}(X,\beta))$ with equality only if $m=n$. Combining with \eqref{inequality: dim constraint}, assumed to be strict if $m=n$, we conclude that $\dim(Z)<\mathrm{dim}(\oM_{g,n} \times X^n)$, which is a contradiction.
    \end{proof}

\begin{proposition}\label{prop2: how to prove SAE}
Let $X$ be a smooth, projective Fano variety. Assume that, for all $g\ge 0$, there exists an integer $C'(X,g)$ with the following property: if $n,\beta$ satisfy $n\ge C'(X,g)$ and the inequality \eqref{inequality: dim constraint}, then conditions (*) and (**) hold for all components $Z\subset \cM_{\Gamma}^{(m,a)}\subset \cM_{g,n}(X,\beta)$. Assume furthermore that we can take $C'(X,0)=3$.

Then, $X$ satisfies SAE.
\end{proposition}

\begin{proof}
For any $n,\beta$, consider the map
$$\tau: \oM_{g,n}(X,\beta) \rightarrow \oM_{g,n} \times X^{n}.$$
We will show that, if \eqref{inequality: dim constraint} holds, and if in addition
\begin{equation*}
    n\ge 
    \begin{cases}
    (r+1)(C'(X,g)+g) &\text{ if }g\ge 1\\
    3 &\text{ if }g\ge 0,
    \end{cases}
\end{equation*}
then every component $Z\subset \cM_{\Gamma}^{(m,a)}\subset \cM_{g,n}(X,\beta)$ fails to dominate $\oM_{g,n} \times X^{n}$, unless $(m,a)=(n,0)$ and the inequality \eqref{inequality: dim constraint} is an equality. This is sufficient to conclude SAE because \eqref{inequality: dim constraint} is weaker than \eqref{dim constraint}.

    Fix a component $Z$ and a general point $f\in Z$. Because $X$ is Fano, we have 
    $$
    f_*([T_i'])\cdot K_X^\vee > 0
    $$
    for all $i=1,\ldots,a$. Therefore, we may assume without loss of generality that $a=0$, by deleting the trees $T_i'$. The value of $n$ stays the same and $\beta \cdot K_X^\vee$ decreases strictly, so \eqref{inequality: dim constraint} still holds.

    If the pushforward under $f$ of each irreducible component of each $T_i$ for $i=1,\ldots,n-m$ is ordinary, then we are done by
    Proposition \ref{prop1: how to prove SAE}. 
    
    Suppose now without loss of generality that $T_1,\ldots,T_s$ with $0<s\leq n-m$ are trees each containing some component whose pushforward via $f$ is not ordinary. Then, by condition (**), we have 
    $$
    \mathrm{deg}(f_*[T_i]) \geq r+1
    $$
    for all $i=1,\ldots,s$. Let 
    $\hat{f}: \hat{C} \rightarrow X$ be the stable map obtained from $f$ by deleting $T_1,\ldots,T_s$ and $x_1,\ldots,x_s$ and let $\widehat{\beta}= \widehat{f}_*[\widehat{C}]$. Notice that $[\widehat{f}]$ belongs to an irreducible component $\widehat{Z}$ of
    $$
    \cM_{\Gamma}^{(m,0)} \subset \ \oM_{g,n-s}(X,\widehat{\beta})
    $$
    
    It is enough to show that $\widehat{Z}$ does not dominate $\oM_{g,n-s} \times X^{n-s}$, provided that $2g-2+(n-s)>0$. This will follow from Proposition \ref{prop1: how to prove SAE}, once we notice the following two facts. First, from condition (**) we have
         $$
      \beta\cdot K_X^\vee - \widehat{\beta}\cdot K_X^\vee = \sum_{i=1}^s f_*[T_i]\cdot K_X^\vee \ge (r+1)s,
     $$
    and so 
        $$
    \mathrm{vdim} (\oM_{g,n-s}(X,\widehat{\beta})) \le \mathrm{vdim} (\oM_{g,n}(X,\beta))-(r+2)s.
    $$
Combining with \eqref{inequality: dim constraint} yields
    $$
    \mathrm{vdim} (\oM_{g,n-s}(X,\widehat{\beta})) < \mathrm{dim}(\oM_{g,n-s} \times X^{n-s}).
    $$
    If $(g,n-s)=(0,2)$, then the same calculation shows that $\mathrm{vdim} (\oM_{0,2}(X,\widehat{\beta}))<\dim(X^2)-1$.
    
    Second, by \eqref{inequality: dim constraint} and condition (**), we have
    $$
    s \leq \frac{\beta\cdot K_X^\vee}{r+1} \le \frac{r}{r+1}\cdot n+\frac{r}{r+1}\cdot (g-1).
    $$
    If $g\ge 1$, then the assumption $n\ge (r+1)(C'(X,g)+g)$ gives $n-s\ge C'(X,g)$, so we conclude by Proposition \ref{prop1: how to prove SAE}, applied with $n-s$ in place of $n$.

   If $g=0$, then $n-s\ge \frac{n+r}{r+1}$, hence $n-s\ge 2$. We conclude again by Proposition \ref{prop1: how to prove SAE} unless $n-s=2$. If $n-s=2$, then all components (including the spine) of $\hat{C}$ are ordinary, so
       $$
    \mathrm{dim}_{\hat{f}}(\oM_{0,2}(X,\widehat{\beta})) = \mathrm{vdim} (\oM_{0,2}(X,\widehat{\beta}))<\dim (X^2),$$
    and $\oM_{0,2}(X,\widehat{\beta})$ fails to dominate $X^2$.
    

\end{proof}



\subsection{Blow-ups of $\P^r$}

Let $X$ be the blow-up of $\P^r$ at $\ell$ general points. The next results will later be useful to deal with conditions (*) and (**).

\begin{lemma}\label{inequality d and ks for ordinary}
Let 
$$
\beta= d\H^\vee + \sum_{i=1}^\ell k_i \E_i^\vee
$$
be a non-zero effective curve in $H_2(X,\Z)$. If $\beta$ is ordinary, then
$$
d \geq \sum_{i \in I} k_i \text{ for all } I \subseteq \{1,\ldots,\ell \} \text{ with } |I| \leq r.
$$
\end{lemma}
\begin{proof}
Fix $I \subseteq \{1,\ldots,r+1\}$ such that $|I| = r$. Let $f:\P^1 \rightarrow X$ be a curve in class $\beta$; because $\beta$ is ordinary, we may assume that the image of $f$ is not contained the strict transform $\Lambda_I$ of the hyperplane in $\P^r$ generated by the points $\pi(\E_i)$ for $i \in I$. In particular, we have
$$
0\le \beta\cdot\Lambda_I=d-\sum_{i \in I} k_i.
$$
as desired.
\end{proof}

\begin{lemma}\label{min degree ordinary curves}
     Let $0 \neq \beta \in H_2(X,\Z)$ be an effective ordinary curve class. Then, $\beta\cdot K_X^\vee \geq 2$.
\end{lemma}
\begin{proof}
    By \cite[Proposition 13]{LP}, every irreducible component of $\oM_{0,1}(X,\beta)$ dominating $X$ has dimension equal to $\beta\cdot K_X^\vee+r-2$, which must be greater than or equal to the dimension of $X$.
\end{proof}

\begin{lemma}\label{h1 nice curve classes}
    Assume that $\ell \leq r+1$. Let $f: C \rightarrow X$ be a map from a smooth genus $g$ curve in class $d \H^\vee + \sum_{i=1}^\ell k_i \E_i^\vee$ with either $d>0$ or $d=k_1=\ldots=k_\ell=0$. Then, $h^1(C, f^*T_X) \leq rg$.
\end{lemma}
\begin{proof}
    When $d=k_1=\ldots=k_\ell=0$, we have an equality. Assume that $d>0$.

    If $\ell\leq r+1$, then 
    $$
    H^0(X,T_X) \otimes \mathcal{O}_X \rightarrow T_X
    $$
    is surjective on $X\smallsetminus \bigcup_{i=1}^\ell \E_i$, and so also 
    $$
    H^0(C,f^*T_X) \otimes \mathcal{O}_C \rightarrow f^*T_X
    $$ 
    is surjective on $C \smallsetminus \bigcup_{i=1}^\ell f^{-1}(\E_i)$, which is non-empty. Choosing $r$ global sections of $H^0(C,f^*T_X)$, linearly independent upon restriction to some point of $C$, yields a morphism
    $$
    \mathcal{O}_C^{\oplus r} \rightarrow f^*T_X
    $$
    which is surjective outside finitely many points on $C$. In particular, the induced map
    $$
    H^1(C,\mathcal{O}_C)^{\oplus r}  \twoheadrightarrow H^1(C,f^*T_X)
    $$
    is surjective, and this gives the stated conclusion.
\end{proof}

\subsection{SAE for $\Bl_q(\P^r)$}\label{Section:SAE for the blow-ups at one point}

Let $X=\Bl_q(\P^r)$ be the blow-up of $\P^r$ at one point. Write $\H,\E$ for the divisor classes of the pulled-back hyperplane class and the exceptional divisor, respectively, and write $\H^\vee,\E^\vee$ for the corresponding dual basis of $H_2(X)$. 

\begin{theorem}\label{SAE for blow-ups at 1 point}
    $X$ satisfies SAE.
\end{theorem}
\begin{proof}
We adopt the setup of Proposition \ref{prop2: how to prove SAE}. The assumption \eqref{inequality: dim constraint} here is 
\begin{equation}\label{eq:dim constraint_bl_pt}
    n\ge \frac{r+1}{r}\cdot d-\frac{r-1}{r}\cdot k-g+1.
\end{equation}

    \begin{enumerate}[label=\underline{Step \arabic*}]
        \item We estimate the quantity 
        \begin{equation*}
        \mathrm{dim}_{[f|_{C_{\spine}}]}(\oM_{g,n}(X,f_*[C_{\spine}]))-\mathrm{vdim}(\oM_{g,n}(X,f_*[C_{\spine}])),
        \end{equation*}
         appearing in \eqref{eq:dim-vdim}, under the assumptions of (*). There are two cases.
        \begin{enumerate}[label=\underline{Case \arabic*}]
            \item Suppose that $f_*[C_{\spine}]= -k'\E^\vee$, with $k'>0$. Notice that, since the pushforward under $f$ of each component of $T_i$ is an ordinary class for $i=1,\ldots,n-m$, by Lemma \ref{inequality d and ks for ordinary}, we must have 
             \begin{equation}\label{bound of k' for blow-up 1 pts}
                 d-(k+k') \geq 0.
             \end{equation} 
            
            On the other hand, we have
            $$
            \oM_{g,n}(X,-k'\E^\vee)=\oM_{g,n}(\P^{r-1},k'),
            $$
            and so by Lemma \ref{h1 nice curve classes},
            \begin{equation*}
            \mathrm{dim}_{[f|_{C_{\spine}}]} \oM_{g,n}(X,-k'\E^\vee) \leq  \mathrm{vdim}(\oM_{g,n}(\P^{r-1},k')) +(r-1)g. \\
            \end{equation*}
            
            Therefore,
            \begin{align*}
           & \mathrm{dim}_{[f|_{C_{\spine}}]} (\oM_{g,n}(X,f_*[C_{\spine}]))-  \mathrm{vdim}(\oM_{g,n}(X,f_*[C_{\spine}])) \\
             \leq & \, \mathrm{vdim}(\oM_{g,n}(\P^{r-1},k')) -\mathrm{vdim}(\oM_{g,n}(X,f_*[C_{\spine}])) + (r-1)g \\
             = & \, [r(k'-g+1)-1]-[(r+1)(-g+1)+ k'(r-1)-1]+(r-1)g \\
             = & \, g-1+k'+(r-1)g \\
             \leq & \, rg-1+d-k. 
            \end{align*}
            Now, toward condition (*), we wish to verify that, for all sufficiently large $n$, the inequality 
            \begin{equation}\label{eq:bl_pt_ineq}
                rg-1+d-k\le n-(g+1) \Leftrightarrow n-d+k \ge (r+1)g
            \end{equation}
 holds. We claim that \eqref{eq:bl_pt_ineq} holds as long as $n\ge (r^2+3r+1)g$. Indeed, if $d\ge r(r+2)g$, then by \eqref{eq:dim constraint_bl_pt}, we have
            \begin{equation*}
                n-d+k\ge \frac{d}{r}+\frac{k}{r}-g+1 \ge (r+1)g.
            \end{equation*}
If $d<r(r+2)g$, then instead we have 
\begin{equation*}
    n-d+k\ge (r^2+3r+1)g-(r^2+2r)g=(r+1)g.
\end{equation*}
            
            \item If instead $f_*[C_{\spine}]=d'\H^\vee + k' \E^\vee$ with $d' \geq k' \geq 0$ and $d'>0$ or $d'=k'=0$, then Lemma \ref{h1 nice curve classes} gives that $h^1(C_{\spine},f^{*}T_X)\le rg$. Therefore,
                \begin{equation*}
    \mathrm{dim}_{[f|_{C_{\spine}}]}(\oM_{g,n}(X,f_*[C_{\spine}])) - \mathrm{vdim}(\oM_{g,n}(X,f_*[C_{\spine}]))\leq h^1(C_{\spine},f^{*}T_X)\le n-(g+1).
    \end{equation*}
    as long as $n\ge rg+g+1$.
        \end{enumerate}
\, \\        
Combining cases 1 and 2, we see that condition (*) holds for 
\begin{equation*}
    n\ge C'(X,g):=\max((r^2+3r+1)g,rg+g+1)
\end{equation*}
for all $g$, and in particular for $n\ge 3$ if $g=0$.
        
        \item We verify condition (**). Here, we require no assumption on $n$. If some tree $T_i$ for $i \in \{1,\ldots,n-m\}$ contains a component $T_i^0$ whose pushforward via $f$ is not ordinary, then it must be 
        $$
        f_*[T_i^0]= -k' \E^\vee \text{ for some }  k'>0 
        $$
        which has degree $k'(r-1) \geq r-1$. Also, $T_i$ has to contain a component $\overline{T_i}$ whose pushforward via $f$ is ordinary and which, by Lemma \ref{min degree ordinary curves}, has degree at least $2$. Therefore, $f_*[T_i]$ has at least degree $r+1$.
    \end{enumerate}
  Therefore the criterion of Proposition \ref{prop2: how to prove SAE} holds, and $X$ satisfies SAE.  
\end{proof}

\subsection{SAE for del Pezzo surfaces}\label{Section: SAE for Del Pezzo surfaces}

\begin{theorem}\label{thm: SAE for Del Pezzo surfaces}
Let $X$ be a del Pezzo surface. Then, $X$ satisfies SAE.
\end{theorem}

\begin{proof}
    We adopt the setup of Proposition \ref{prop2: how to prove SAE} again. 
    \begin{enumerate}[label=\underline{Step \arabic*}]
        \item First, we consider condition (*). If $m \geq  g+1$, then by \cite[Proposition 13]{LP}, we have $h^1(C_{\spine},f^*T_X)=0$, and so 
        $$
        \mathrm{dim}_{[f|_{C_{\spine}}]}(\oM_{g,n}(X,f_*[C_{\spine}]))=  \mathrm{vdim}(\oM_{g,n}(X,f_*[C_{\spine}]))
        $$
        and condition (*) holds. 
        
        If instead $m < g+1$, then $n-m > n-g-1$. Observe that, for all $i=1,\ldots,n-m$, the pushforward of some component $\overline{T_i}$ of $T_i$ must be an ordinary class different from $0$. Therefore, by Lemma \ref{min degree ordinary curves}, we have 
        $$
        f_*([T_i])\cdot K_X^\vee \geq 2 \text{ for } i=1,\ldots,n-m,
        $$ 
        and so
        \begin{equation} \label{bound deg betaspine}
        f_*[C_{\spine}]\cdot K_X^\vee \leq \beta\cdot K_X^\vee-2(n-m)  < 2( (g+1)+g-1),
        \end{equation}
        where in the second inequality we used \eqref{inequality: dim constraint}. In particular, when $g=0$, the right hand side is $0$, yielding a contradiction. Therefore, for $g=0$, it must be that $m \geq 1= g+1$ and $h^1(C_{\spine},f^*T_X)=0$. 
        
        For $g$ arbitrary, \eqref{bound deg betaspine} implies that $f_*[C_{\spine}]$ has uniformly bounded degree against $K_X^\vee$, so there are only finitely many possibilities for $f_*[C_{\spine}]$ by the polyhedrality of the effective cone of $X$. Therefore, $h^1(C_{\spine},f^*T_X)$ is bounded by a constant $b[X,g] \in \Z_{\geq 0}$ (depending only on $X$ and $g$), and we have
        $$
        \mathrm{dim}_{[f|_{C_{\spine}}]}(\oM_{g,n}(X,f_*[C_{\spine}]))-  \mathrm{vdim}(\oM_{g,n}(X,f_*[C_{\spine}])) \leq  b[X,g].
        $$
        This implies that condition (*) holds, for all $n$ sufficiently large (when $g$ is arbitrary), and for all $n\ge 3$ (when $g=0$).
        \item We now deal with condition (**). Let $T_i$ be a tree containing a component whose pushforward under $f$ is not ordinary. It still must contain a component $\overline{T_i}$ such that $f_*([\overline{T_i}])$ is an ordinary curve (and so of degree at least $2$ by Lemma \ref{min degree ordinary curves}), and at least one more irreducible component, of degree at least $1$ because $X$ is Fano. Therefore, $\mathrm{deg}(f_*([T_i]) \geq 3=r+1$  and we are done. 
    \end{enumerate}
\end{proof}

\subsection{SAE in dimension $3$}
Let $\pi: X \rightarrow \P^3$ be the blow-up of $\P^3$ at $\ell \leq 4$ general points. Write $\H,\E_1,\ldots,\E_\ell$ for the divisor classes of the pulled-back hyperplane class and the exceptional divisors, respectively, and write $\H^\vee,\E^\vee_1,\ldots,\E^\vee_\ell$ for the corresponding dual basis of $H_2(X)$.

In this section, we prove that $X$ satisfies SAE (Theorem \ref{SAE for blow-up of P3 at 4 pts}). Note that such $X$ are not Fano; this is, to our knowledge, the first example of a non-Fano variety whose Tevelev degrees are asymptotically enumerative. 

\subsubsection{Low degree effective curves in $\P^3$}
We begin by recalling a result about the effective cone of $X$.

\begin{lemma}\cite[Proposition 4.1]{CLO}
 The effective cone of $X$ is linearly generated by $1$-dimensional linear spaces in the exceptional divisors and the strict transforms of $1$-dimensional linear subspaces of $\P^3$, possibly passing through the $\ell$ blown up points.
\end{lemma}

\begin{corollary}
 The anticanonical class $K_X^\vee=4 \H -2 (\E_1 +\cdots+\E_\ell)$ is nef.
\end{corollary}

\begin{remark}
    
   For any curve class $\beta$, the intersection number $\beta \cdot K_X^\vee$ is always even; we will use this fact later.
\end{remark}

\begin{lemma}\label{lemma: degree 0 curves in P3}
    Let $C \subset X$ be an irreducible curve of geometric genus 0 and with non-positive degree against $K_X^\vee$. Then, $\beta:=[C] \in H_2(X,\Z)$ is the class of the strict transform of a line through two of the $q_j$. (In particular, $C$ has degree 0 against $K_X^\vee$.)
\end{lemma}
\begin{proof}
    Write 
    $$
    \beta=d \H^\vee + \sum_{j=1}^\ell k_j \E_j^\vee;
    $$
    we may assume $d,k_1,\ldots,k_\ell \geq 0$.
    
    Let $\mathbb{T} \subset \P^3$ be a torus acting on $\P^3$ in such a way that the $q_j$ are the torus-fixed points. We also denote by $\mathbb{T} \subset X$ the pullback to $X$, so that $\mathbb{T}$ also acts on $X$. We claim that $C$ lies in the strict transform $\Lambda$ of some torus-fixed plane in $\P^3$. Suppose instead that this is not the case. Then, we will show that $\beta$ would be ordinary, and so by Lemma \ref{min degree ordinary curves}, would have anti-canonical degree at least 2, a contradiction.
    
    Indeed, let $\eta: \widetilde{C} \rightarrow C$ the normalization map and let $p \in \widetilde{C}$ be such that $\eta(p)$ does not belong to the strict transform of any torus fixed hyperplane of $\P^3$.  Then, the morphism 
    \begin{align*}
    \mathbb{T} \times \widetilde{C} & \rightarrow X \\
    (t, x)& \mapsto t\cdot\eta(x)
    \end{align*}
    would yield a dominant map
    $$
    \mathbb{T} \rightarrow \cM_{0,1}(X,[C]) \xrightarrow{\mathrm{ev}_p} X
    $$
    implying that $[C]$ is ordinary.
    
    Therefore, the curve $C$ lies in the strict transform $\Lambda\subset X$ of some torus-fixed plane in $P\subset \P^3$, where $P$ contains $\ell'\le 3$ of the points $q_j\in\P^3$. Without loss of generality, we assume that these points are $q_1,\ldots,q_{\ell'}$. In particular, $\Lambda$ is isomorphic to the blow-up of $\P^2$ at $q_1,\ldots,q_{\ell'}$. Furthermore, the class $\beta\in H_2(X,\Z)$ is the pushforward of a class $\beta_\Lambda \in H_2(\Lambda,\Z)$. 
    
    We claim now that $C$ must be equal to the proper transform of a line in $P$ through two of the points $q_j$, where $j\le \ell'$. Suppose that this were not the case. Then, we would have
    \begin{equation}\label{impossible equation}
    d\geq k_i+k_j \text{ for all } 1 \leq i <j \leq 3
    \end{equation}
(we take $k_j=0$ for $j>\ell$ if $\ell\le 2$), whence
    \begin{align*}
        1 &\geq K_X^\vee\cdot\beta \\
        &=4d-2(k_1+k_2+k_3)\\
        &=4d -(k_1+k_2)-(k_1+k_3)-(k_2+k_3) \\
        &\geq d.
    \end{align*}
    On the other hand, the degree against $K_X^\vee$ of the classes $\H^\vee+\E_i^\vee$ and $-\E_j^\vee$ is $2$, so we find no possibilities for $\beta$. We have reached a contradiction, completing the proof.
\end{proof}

\begin{lemma}\label{lemma: degree 2 curves in P3}
    Let $\beta \in H_2(X,\Z)$ be an effective ordinary curve class such that $K_X^\vee\cdot C =2$. Then, $\beta =\H^\vee +E^\vee_j$ for some $j\in \{1,\ldots,\ell \}$.
\end{lemma}
\begin{proof}
Write 
$$
\beta= d \H^\vee + \sum_{j=1}^\ell k_j \E_j^\vee.
$$
Then, by hypothesis,
$$
2=\beta\cdot K_X^\vee =4 d-2(k_1+\cdots+k_\ell)
$$
and, since $\beta$ is ordinary, by Lemma \ref{inequality d and ks for ordinary} we have
$$
d \geq \sum_{j \in J} k_j \text{ for all } J \subseteq \{1,\ldots,\ell\} \text{ with } |J|=3. 
$$

Therefore, we have $4d \geq 3(k_1+\cdots+k_\ell)$ and so
$$
2-(k_1+\cdots+k_\ell) =4d-3(k_1+\cdots+k_\ell) \geq 0
$$
and we obtain $k_1+\cdots+k_\ell \in \{0,1,2\}$, from which the lemma follows easily.
\end{proof}

The previous lemmas motivate the following definition.

\begin{definition}
Let $0 \neq \beta \in H_2(X,\Z)$ be an effective curve class. We say that:
   \begin{enumerate}  
        \item[$\bullet$] $\beta$ is \textbf{exceptional} if $\beta=- k \E_j^\vee$ for some $k>0$ and $j\in \{1,\ldots,\ell \}$,
        \item[$\bullet$] $\beta$ is a \textbf{fixed line} if it is a positive multiple of the strict transform of a line in $\P^3$ through two of the blown-up points, and
        \item[$\bullet$] $\beta$ is a \textbf{pivoting line} if it is not a fixed line and is the strict transform of a line in $\P^3$ through one of the blown-up points.
   \end{enumerate}
\end{definition}

We remark that exceptional classes and fixed lines are not ordinary, but pivoting lines are.

\subsubsection{Proof of SAE}

Let $\pi:X \rightarrow \P^3$ be the blow-up at $\ell \leq 4$ points $q_1,\ldots,q_\ell$. Fix a genus $g \geq 0$ and a non-zero effective curve class $\beta$. We will follow the setting of \S\ref{Section: How to prove SAE for Fano varieties}, but need to make suitable modifications now that $X$ is not Fano.

For integers $a\ge0$ and $s,m\ge0$ such that $s+m\le n$, define
$$
\cM_{\Gamma}^{(s,m,a)}
 \subseteq \oM_{g,n}(X,\beta)
 $$
 to be the locally closed locus parametrizing maps whose domain $C$ has the following topological type. $C$ contains a smooth \emph{spine} component $C_{\spine}$ of genus $g$, containing the marked points $p_{n-m+1},\ldots,p_n$. Attached to $C_{\spine}$, we have trees of rational curves $T'_1,\ldots,T'_a$, as well as trees of rational curves $T_1,\ldots,T_{n-m}$, such that the tree $T_{i}$ contains the marked point $p_i$ for $i=1,\ldots,n-m$. Furthermore, the integer $s$ is such that $T_i$ contains a non-ordinary component if and only if $i\le s$. See also Figure \ref{singular_stable_map 2}.
 \begin{figure}
	\begin{center}
		\begin{tikzpicture} [xscale=0.36,yscale=0.36]
                \draw [ultra thick, black] (-10,7) to (19,7);
                \node at (15,7) {$\bullet$};
                \node at (15,8) {$p_{n-m+1}$};
                \node at (18,7) {$\bullet$};
                \node at (18,8) {$p_n$};
                \node at (19,6) {$C_{\spine}$};
                 \draw [thick, dotted] (16,6) -- (17,6);
                \draw [ultra thick, black] (1,8) to (1,3);
                \draw [ultra thick, black] (0,5) to (2,2);
                \draw [ultra thick, black] (2,3) to (0,0);
                \node at (0,-1) {$T_1$};
                \node at (1,5.5) {$\bullet$};
                \node at (2,5.5) {$p_1$};
                \draw [ultra thick, dotted] (3,4) -- (4.5,4);
                \draw [ultra thick, black] (5,8) to (6,3);
                \draw [ultra thick, black] (6,5) to (5,1);
                \node at (5,0) {$T_s$};
                \node at (5.3,2) {$\bullet$};
                \node at (4.8,2.5) {$p_s$};
                \draw [ultra thick, black] (8.4,8) to (8.4,3);
                \node at (8.4,4) {$\bullet$};
                \node at (7.2,4) {$p_{s+1}$};
                \node at (8,2) {$T_{s+1}$};
                \draw [ultra thick, dotted] (9,4) -- (10.5,4);
                \draw [ultra thick, black] (13,8) to (11,3);
                \draw [ultra thick, black] (11,6) to (14,1);
                \draw [ultra thick, black] (14,3) to (12,0);
                \draw [ultra thick, black] (11.6,6) to (16,4);
                \node at (13,2.65) {$\bullet$};
                \node at (11.6,2.55) {$p_{n-m}$};
                \node at (12,-1) {$T_{n-m}$};
                \draw [ultra thick, black] (-3.5,8.5) to (-1,2);
                \draw [ultra thick, black] (-2,6) to (-3,2);
                \draw [ultra thick, black] (-1,5) to (-3,1);
                \node at (-3,0) {$T_a'$};
                \draw [ultra thick, black] (-6,8) to (-6,1);
                \draw [ultra thick, black] (-5,5) to (-7,1);
                \draw [ultra thick, black] (-1,5) to (-3,1);
                \node at (-6,0) {$T_1'$};
                \draw [ultra thick, dotted] (-5,4) -- (-3.5,4);
                \draw [ decorate,decoration={brace,amplitude=5pt,mirror,raise=1pt}] (-7,-1.5) -- (-2,-1.5) ;
                \node at (-5,-3) {\scriptsize $a$ rational trees};
                \node at (-5,-4){\scriptsize without markings};
                \node at (-5,-5) {\scriptsize on them};
                \draw [ decorate,decoration={brace,amplitude=5pt,mirror,raise=1pt}] (-0.9,-1.5) -- (6,-1.5) ;
                \node at (3,-3) {\scriptsize $s$ marked rational };
                \node at (3,-4){\scriptsize trees with at};
                \node at (3,-5) {\scriptsize least one non-};
                \node at (3,-6){\scriptsize ordinary component};
                \draw [ decorate,decoration={brace,amplitude=5pt,mirror,raise=1pt}] (7,-1.5) -- (17,-1.5) ;
                \node at (12.5,-3) {\scriptsize$n-m-s$ marked ratio- };
                \node at (12.5,-4){\scriptsize nal trees with only};
                \node at (12.5,-5) {\scriptsize ordinary components};
		\end{tikzpicture}
	\end{center}
	\caption{Topological type of the domain curves of points in $\cM_\Gamma^{(m,s,a)}$\label{figure-2}\label{singular_stable_map 2} }
\end{figure}
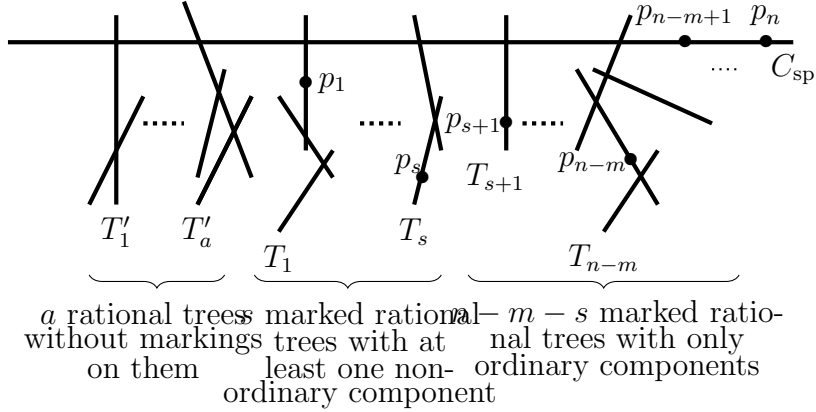 

Again, we assume that if $(n-m)+a=0$, that is, if $C$ is irreducible, then  \eqref{inequality: dim constraint} is a \emph{strict} inequality.
 
 Let $Z$ be an irreducible component of $\cM_{\Gamma}^{(s,m,k)}$. We will prove the following.
 
 \begin{proposition}\label{non-dominance for blow-ups of P3}
Assume that, for all $g\ge 0$, there exists an integer $C'(X,g)$ with the following property. Suppose that $n,\beta$ satisfy $n\ge C'(X,g)$ and that the inequality 
 \begin{equation} \label{inequality: dim constraint dim3}
 \mathrm{vdim} (\oM_{g,n}(X,\beta)) \leq \mathrm{dim}(\oM_{g,n} \times X^n)
 \end{equation}
holds. For any $s,m,k$, let $Z\subset \cM_{\Gamma}^{(s,m,k)}\subset \oM_{g,n}(X,\beta)$ be an irreducible component. Suppose further that if $(m,a)=(n,0)$, then \eqref{inequality: dim constraint dim3} is a \emph{strict} inequality. Then, $Z$ fails to dominate $\oM_{g,n}\times X^n$.
 \end{proposition}
 
The inequality \eqref{inequality: dim constraint dim3} is equivalent to $\beta\cdot K_X^\vee\le 3(n+g-1)$. As an immediate corollary to Proposition \ref{non-dominance for blow-ups of P3}, we obtain the following.

 \begin{theorem}\label{SAE for blow-up of P3 at 4 pts}
    $X$ satisfies SAE.
 \end{theorem}
 
 We prove Proposition \ref{non-dominance for blow-ups of P3} by first making several reductions. Let $[f] \in Z$ be a general point.

\vspace{5pt}
\noindent  \textbf{Reduction 1:} \emph{It is enough to prove Proposition \ref{non-dominance for blow-ups of P3} for $a=0$}.
\begin{proof}
Because $f_*[T_i']\cdot K_X^\vee \geq 0$, this is clear when $m<n$, as we may simply delete the components $T'_i$, so assume that $m=n$. For the same reason, we can also assume that $a=1$, that $T_1'$ is irreducible, and that $f_*[T_1']\cdot K_X^\vee =0$. It follows from Lemma \ref{lemma: degree 0 curves in P3} and stability that 
$$
f_*[T_1']=m (\H^\vee+\E_j^\vee+ \E_{j'}^\vee) \text{ for some } m\in \Z_{> 0} \text{ and } 1 \leq j < j' \leq \ell.
$$

Let $\L$ be the line in $\P^3$ through $\pi(\E_j)$ and $\pi(\E_{j'})$, and $\rho: \widetilde{X} \rightarrow X$ be the blow-up of $X$ along the strict transform of $\L$. Call $\overline{\beta}=\beta-f_*[T_1']$ and let $\overline{Z} \subset \cM_{g,n}(X,\overline{\beta}) $ be the irreducible component to which $[f|_{C_{\spine}}]$ belongs. From $\overline{Z}$ we get an irreducible substack
$$
\widetilde{Z} \subseteq \cM_{g,n}(\widetilde{X},\widetilde{\beta})
$$
dominating $\overline{Z}$. Here, $\widetilde{\beta} \in H_2(\widetilde{X},\Z)$ is an effective curve class with the property that
$$
\widetilde{\beta} \cdot \E >0
$$
where $\E $ is the class of the exceptional divisor in $\widetilde{X}$. 

If $Z$ dominates $X^n \times \cM_{g,n}$, then $\widetilde{Z}$ dominates $\widetilde{X}^n \times \cM_{g,n}$, and applying \cite[Proposition 13]{LP} twice, we obtain
$$
\mathrm{dim}(Z)= n+ \beta\cdot K_X^\vee > n+\widetilde{\beta}\cdot K_{\widetilde{X}}^\vee \geq \mathrm{dim}(\widetilde{Z}),
$$
contradicting the fact that $\widetilde{Z}$ dominates $Z$.
\end{proof}

Assume henceforth that $a=0$.

\vspace{5pt}
\noindent  \textbf{Reduction 2:} \emph{It is enough to prove Proposition \ref{non-dominance for blow-ups of P3} for $s=0$}.

\begin{proof}
If, for some $i \in \{1,\ldots,s\}$, we have $f_*[T_i]\cdot K_X^\vee \geq 4= \mathrm{dim}(X)+1$, then arguing as in the proof of Proposition \ref{prop2: how to prove SAE} above, we can delete $T_i$ from the domain. 

Thus, we may assume that, for every $i=1,\ldots,s$, we have $f_*[T_i]\cdot K_X^\vee \leq 2$ (recall that curves on $X$ always have even anti-canonical degree). Let
\begin{equation}\label{dec irr T_i}
    T_i = \bigcup_{j=0}^{k_i} T_i^j
\end{equation}
be the decomposition in irreducible components of $T_i$. Then, by Lemmas \ref{lemma: degree 0 curves in P3} and \ref{lemma: degree 2 curves in P3}, $f_*[T_i^j]$ is a fixed line for all but one $j \in \{ 0,\ldots.,k_i\}$, say except for $j=0$, and $f_*[T_i^0]$ is a pivoting line $\H^\vee+\E_h^\vee$ for some $h\in \{1,\ldots,\ell\}$. 

Now, for $[f]\in Z$ such that $f(p_i)=x_i$, we need $x_i \in f(T_i^0)$. Thus, the image $f(T_i^0)$ is the strict transform of the unique line between $q_h$ and $x_i$, while for $j\ge1$, the reduced image $f(T_i^j)^{\mathrm{red}}$ is either a point or the strict transform of the unique line in $\P^3$ through two of the points $q_1,\ldots,q_\ell$. By assumption, there is at least a $j$ for which $[f(T_i^j)^{\mathrm{red}}]$ is a fixed line. However, for general $x_i$, two such curves do not meet in $X$. Thus, $Z$ cannot dominate $\oM_{g,n}\times X^n$ unless $s=0$.
\end{proof}

Assume henceforth that $a=s=0$.

\vspace{5pt}
\noindent  \textbf{Reduction 3:} \emph{We can assume that each $T_i$ is irreducible and that $f_*[T_i]$ is a pivoting line for $i=1,\ldots.,n-m$.}
\begin{proof}
    If $f_*[T_i]\cdot K_X^\vee \geq 4$, then arguing as in the proof of Proposition \ref{prop2: how to prove SAE}, we can delete $T_i$. By Lemma \ref{min degree ordinary curves}, we can assume that if \eqref{dec irr T_i} is the decomposition in irreducible components of $T_i$, then $f(T_i^j)$ is a point unless $j=0$, and $f(T_i^0)$ is the strict transform of a line through one of the points $q_1,\ldots,q_\ell$. In particular, replacing $T_i$ with $T_i^0$, it is enough to prove Proposition \ref{non-dominance for blow-ups of P3} when $T_i=T_i^0$.
\end{proof}

Assume henceforth that each $T_i$ is irreducible and that $f_*[T_i]$ is a pivoting line for $i=1,\ldots.,n-m$.

\vspace{5pt}
\noindent  \textbf{Reduction 4:} 
It is enough to prove the following analog of condition (*).
\begin{enumerate}
    \item[(*')] The following inequality holds.
    \begin{equation*}
    \mathrm{dim}_{[f|_{C_{\spine}}]}(\oM_{g,n}(X,f_*[C_{\spine}])) - \mathrm{vdim}(\oM_{g,n}(X,f_*[C_{\spine}]))\leq n-(g+1).
    \end{equation*}
\end{enumerate}

\begin{proof}
    The proof is the same as that of Proposition \ref{prop1: how to prove SAE}.
\end{proof}

\begin{proof}[Proof of Proposition \ref{non-dominance for blow-ups of P3}]
It is enough to prove that condition (*') holds whenever $n$ is large (depending on $X$ and $g$) and whenever $g=0$ and $n \geq 3$. There are two cases.
\begin{enumerate}[label=\underline{Case \arabic*}]
    \item Suppose that $f_*[C_{\spine}]= -k'\E_h^\vee$ with $k'>0$, for some $h \in \{1,\ldots, \ell \}$. Then, $f(T_i)$ must be the strict transform of some line in $\P^3$ through $q_h$, and $C_{\spine}$ cannot contain any marked point (i.e., $m=0$). In particular, we have
    $$
    \beta=f_*[C]=n \H^\vee + (n-k') \E_h^\vee,
    $$
    and so by \eqref{inequality: dim constraint dim3},
    $$
    3(n+g-1)\ge\beta\cdot K_X^\vee =4n-2(n-k')=2 n+ 2k',
    $$
    from which we deduce 
    $$
    k'\le \frac{n}{2}+\frac{3}{2}(g-1).
    $$
    Proceeding as in Theorem \ref{SAE for blow-ups at 1 point}, we find
                \begin{align*}
           & \mathrm{dim}_{[f|_{C_{\spine}}]}(\oM_{g,n}(X,f_*[C_{\spine}])) - \mathrm{vdim}(\oM_{g,n}(X,f_*[C_{\spine}])) \\
             \leq & \, \mathrm{vdim}(\oM_{g,n}(\P^{2},k')) -\mathrm{vdim}(\oM_{g,n}(X,f_*[C_{\spine}])) + 2g \\
             = & \, [3(k'-g+1)-1]-[4(-g+1)+2k'-1]+2g \\
             = & \, 3g-1+k' \\
             \leq & \, \frac{9}{2}g-\frac{5}{2}+\frac{1}{2}n
            \end{align*}
    which implies the required inequality (*') if $n\ge 11g-7$ or if $g=0$ and $n\ge 3$.
    \item If $f_*[C_{\spine}]= d' \H^\vee+k_1'\E_1^\vee+\cdots+k_\ell'\E_\ell^\vee$ with $d'>0$ or it is the $0$ class, then condition (*') for sufficiently large $n$ follows from Lemma \ref{h1 nice curve classes}.
\end{enumerate}
This concludes the proof.
\end{proof}

\section{Geometric counts}\label{Explicit computations} 

\subsection{Integral formula on $\P$}\label{section: proof of integral formula}

In this section, we prove Proposition \ref{thm:integral formula}. We first recall the setup of \S\ref{geom_intro}.

Let $\pi:X \rightarrow \P^r$ be the blow-up of $\P^r$ at $\ell \leq r+1$ torus-fixed points, and let
$$
\beta= d \H^\vee + \sum_{i=1}^\ell k_i \E_i^\vee\in H_2(X,\Z)
$$
be an effective curve class.

Let $S:=\J(C) \times \Sym^{k_1}(C) \times \ldots  \times \Sym^{k_\ell}(C)$, and let $\nu:\P\to S$ be the projective bundle whose fiber over $(\cL,D_1,\ldots,D_\ell)\in S$ parametrizes $(r+1)$-tuples of sections 
$$f_i\in H^0\left(C,\cL\left(- \sum_{j \neq i} D_j\right)\right)$$ up to simultaneous scaling. We will write $$\cL':=\cL\left(- \sum_{j=1}^{\ell} D_j\right).$$

Recall also the notation $$\widetilde{\H}= c_1( \mathcal{O}_\P(1)) -\eta_1-\cdots-\eta_\ell \in H^*(\P,\Z)$$
where $\eta_i$ is the pullback from $\Sym^{k_i}(C)$ of the class of the divisor $N_i=\{D: D-p \geq 0 \} $ (here $p$ is any fixed point in $C$).

Fix general points $x_1,\ldots,x_n \in X$. We can assume that $x_i$ lies in the locus where $\pi$ is an isomorphism and identify  $x_i \in \P^r$, so that we can write
$$
x_i=[x_{i,1}:\cdots:x_{i,r+1}].
$$
We wish to cut out subschemes of $V(x_i)\subset\P$ corresponding to the conditions $f(p_i)=x_i$, but care must be taken when the $p_i$ vanishes at some of the $f_j$. For a given point $[f_1:\cdots:f_{r+1}]$ of $\P$, in order for $f(p_i)=x_i$, we need
\begin{equation}\label{def V(xi)}
    x_{i,j}\cdot f_h(p_i)-x_{i,h}\cdot f_j(p_i)=0
\end{equation}
for all $j,h$. In order for the equation \eqref{def V(xi)} to make sense, we need to regard $f_h,f_j$ as sections of the same line bundle; the line bundle of minimal degree for which this is possible is $\mathcal{L}'(D_h+D_j)$.

More generally, for all $i=1,\ldots,n$ and $a\le r$, we define $V(x_i)_a$ to be the closed subscheme of points $$(\mathcal{L},D_1,\ldots,D_\ell,\{f=[f_1:\cdots:f_{r+1}]\})\in\P$$
cut out by the equations \eqref{def V(xi)} for all $h,j$ with $1 \leq h < j \leq a+1$, where we interpret $f_h$ and $f_j$ as sections of $\mathcal{L}'(D_h+D_j)$, and then restrict to $p_i$. That is, $V(x_i)_a$ is the locus of $f$ where the first $a+1$ coordinates of $f(p_i)$ agree with those of $x_i$, where the $h$-th and $j$-th coordinates of $f(p_i)$ and $x_i$ are compared by regarding $f_h,f_j$ as sections of $\mathcal{L}'(D_h+D_j)$.

Write $V(x_i)=V(x_i)_r$, and define now
$$
V(x_1,\ldots,x_n)= \bigcap_{i=1}^n V(x_i).
$$

Theorem \ref{thm:integral formula} is implied by the following two statements.

\begin{proposition}\label{class of W(xi)}
    For all $i=1,\ldots,n$, we have  
    $$
    [V(x_i)_a]=\widetilde{\H}^a+\sigma_1 \widetilde{\H}^{a-1}+\cdots+ \sigma_{a} \text{ in } H^{2a}(\P).
    $$ 
    where $\sigma_i=\sigma_i(\eta_1,\ldots,\eta_{a+1})$ for $i=1,\ldots,a$.
\end{proposition}

\begin{proposition}[Transversality]\label{Understanding V(x1,...,xn)}
    The scheme $V(x_1,\ldots,x_n)$ is reduced and $0$-dimensional. Moreover, its $\C$-points are in bijection with the set of maps enumerated in $\mathsf{Tev}^X_{g,n,\beta}$. (In particular, the $f_j$, when regarded as sections of $\cL$, do not simultaneously vanish at any $p\in C$.)
\end{proposition}

The proof of transversality is deferred to the next section. Here, we prove Proposition \ref{class of W(xi)}. Throughout the rest of this section, we work with a fixed $i$, and assume that $x_i=[1:\cdots:1]$. Also, by setting $\eta_i=0$ for $i>\ell$, we assume that $\ell=r+1$. 

\begin{lemma}\label{lem:V_codim}
    For all $a\le r$, the subscheme $V(x_i)_a\subset\P$ is generically reduced and irreducible of codimension $a$.
\end{lemma}

\begin{proof}
We first prove the lemma upon restriction to the open set $\P_0\subset\P$ where $p_i\notin D_v$ for all $v=1,2,\ldots,r+1$. In this case, the desired statement is already true upon restriction to the fibers of the projective bundle $\P$. That is, for fixed $\cL\in\J C$ and $D_v\in\Sym^{k_v}C$ not containing $p_i$, the restriction of $V(x_i)_a$ to the fiber of $\P$ over $(\cL,D_v)$ is reduced and irreducible of codimension $a$. To see this, we write
$$\P':=\P\left(\bigoplus_{t=1}^{r+1}H^0(C,\cL'(D_t))\right),$$
for the fiber of $\P$ over $(\cL,D_v)$. Then, the restriction of $V(x_i)_a$ to $\P'$ is the linear space cut out by the $a$ independent linear equations $f_1(p_i)=f_{v}(p_i)$ for $v=2,3,\ldots,a+1$, where we may as well consider $f_1,\ldots,f_{a+1}$ all as sections of the same line bundle $\cL$.

We now claim that $\P_0\cap V(x_i)_a$ is dense in $V(x_i)_a$, which will prove the lemma. We show that any point $$(\mathcal{L},D_1,\ldots,D_{r+1},\{f=[f_1:\cdots:f_{r+1}]\})\in V(x_i)_a$$ is a limit of points in $\P_0\cap V(x_i)_a$. 

Let $B$ be the spectrum of a discrete valuation ring, and let $\pi:C\times B\to B$ be the projection. Let $\widetilde{D}_1,\ldots,\widetilde{D}_{r+1}$ be divisors on $C\times B$ restricting to $D_1,\ldots,D_{r+1}$ on the special fiber, and such that $p_i\notin\widetilde{D}_1,\ldots,\widetilde{D}_{r+1}$ upon restriction to the generic fiber. For each $t$, we construct a section $$\widetilde{f}_t\in H^0(C\times B,\cL'(\widetilde{D}_t))\cong
H^0(B,\pi_{*}(\cL'(\widetilde{D}_t)))$$
restricting to $f_t$ on the special fiber, and for which the restrictions of $\widetilde{f}_t$ to the $\{p_i\}\times B$ are all equal to each other, when regarded as sections of $H^0(B,\cL\vert_{p_i})$. By the assumption that $p_i\notin\widetilde{D}_1,\ldots,\widetilde{D}_{r+1}$ on the generic fiber, the $\widetilde{f}_t$ and $\widetilde{D}_t$ define a $K(B)$-point of $\P_0\cap V(x_i)_a$, as needed.

The construction is as follows. First, consider the restriction map
$$\gamma:H^0(C\times B,\cL'(\widetilde{D}_1))\to H^0\left(C\times B,\left.\cL'(\widetilde{D}_1)\right\vert_{0}\right)\cong H^0\left(C,\cL'(D_1))\right),$$
where $0\in B$ is the closed point. We have $H^1(C\times B,\cL'(\widetilde{D}_1)(-C\times 0))=0$ by Cohomology and Base Change and the Leray spectral sequence, so $\gamma$ is surjective, and we may define $\widetilde{f}_1$ to be any section extending $f_1$.

Similarly, the restriction map $$\gamma':H^0(C\times B,\cL'(D_1))\to H^0\left(C\times B,\left.\cL'(D_1)\right\vert_{(C\times 0)\cup (p_i\times B)}\right)$$ is surjective, because $H^1(C\times B,\cL'(D_1)(-(C\times 0)-(p_i\times B)))=0$. Note here that in order to ensure that $R^1\pi_{*}(\cL'(D_1)(-(C\times 0)-(p_i\times B)))=0$, we need the right hand side of \eqref{stronger inequality d, ks} to be $2g-1$, rather than $2g-2$. Therefore, we can define all other $\widetilde{f}_t$ extending $f_t$ \emph{and} with $\widetilde{f}_1(p_i)=\widetilde{f}_t(p_i)$ in $\mathcal{L}$.
\end{proof}

For all $t=1,2,\ldots,r+1$, let $W_t\subset\P$ denote the divisor cut out by the equation $f_t(p_i)=0$, where $f_t$ is viewed as a section of $\cL'(D_t)$. Similarly, for all $t=1,2,\ldots,r$, let $W_{t,t+1}$ denote the locus on $\P$ of points satisfying $f_{t}(p_i)-f_{t+1}(p_i)=0$, where $f_t-f_{t+1}$ is viewed as a section of $\cL'(D_t+D_{t+1})$. Condition \eqref{stronger inequality d, ks} implies that $W_t$ and $W_{t,t+1}$ are indeed divisors.

\begin{lemma}\label{div classes}
We have $$[W_t]=\widetilde{\H}+\eta_{a}$$ and $$[W_{t,t+1}]=\widetilde{\H}+\eta_{a}+\eta_{a+1}$$ in $H^2(\P)$.
\end{lemma}

\begin{proof}
    We prove the second statement; the first is similar. The locus $W_{t,t+1}$ is cut out by a tautological section of the line bundle 
    $$
    \cO_\P(1)\otimes  \nu_{*}\left(\left.\mathcal{P}\left(-\sum_{v \neq t,t+1} \mathcal{D}_v\right)\right\vert_{p_i}\right)$$
    where $\nu:C\times S\to S$ is the projection. 
    Note now that 
    $$
    c_1\left(  \nu_{*}\left(\left.\mathcal{P}\left(-\sum_{v \neq t,t+1} \mathcal{D}_v\right)\right\vert_{p_i}\right) \right) =c_1 \left( \bigotimes_{v\neq t,t+1}\cO_{\Sym^{k_v}C}(-N_v)) \right) \in H^2(\P)
    $$
    
    This follows from the fact that 
    $$
    c_1(\mathcal{P}\vert_{\J(C) \times \{p\}} ) =0 \in H^2(\J(C))
    $$
    for all $p \in C$ and the isomorphism
    $$\mathcal{O}(-\mathcal{D}_v)\vert_{\Sym^{k_v}(C) \times \{p_i\}} \cong \cO_{\Sym^{k_v}C}(-N_v) \text{ for }v=1,...,r+1. 
    $$
    Therefore, the desired class is
    $$
    \H-\sum_{v \neq t,t+1} \eta_v= \widetilde{\H}+\eta_t+\eta_{t+1}\in H^2(\P).
    $$
\end{proof}

\begin{lemma}\label{recursion for [V]a}
Let $[V]_a\in H^{2a}(\P)$ be the class of $V(x_i)_a\subset\P$. Then, for all $a\le r$, we have
\begin{equation}
        [V]_{a}=[V]_{a-1}(\widetilde{\H}+\eta_a+\eta_{a+1}) -[V]_{a-2}(\widetilde{\H}+\eta_{a}) \eta_{a},
        \end{equation}
where we set $[V]_{-1}=0$ by convention.
\end{lemma}

\begin{proof}
The claim follows from the following three statements, after applying Lemma \ref{div classes}.
\begin{enumerate}
\item[(1)] We have a set-theoretic equality
$$(V_{a-1}\cap W_{a,a+1})=V_a\cup(V_{a-2}\cap W_a\cap N_a)$$ in $\P$.
\item[(2)] The subschemes $V_{a-1}\cap W_{a,a+1}$ and and $V_{a-2}\cap W_a\cap N_a$ of $\P$ (in addition to $V_a$) are generically reduced and irreducible of codimension $a$.
\item[(3)] $V_{a-1}\cap W_{a,a+1}$ and $V_{a-2}\cap W_a\cap N_a$ are not equal.
\end{enumerate}

Consider (1). The inclusion $\supset$ is straightforward. A point in $f\in V_{a-1}\cap W_{a,a+1}$ satisfies \eqref{def V(xi)} whenever $1\le h<j\le a$, and in addition for $(h,j)=(a,a+1)$. If it is not the case that $f\in V_{a}$, then \eqref{def V(xi)} must fail for $j=a+1$ and some $h<a$; without loss of generality, take $h=1$. 

The sections $f_1\in\cL'(D_1)$ and $f_{a+1}\in\cL'(D_{a+1})$ are not equal at $p_i$ when regarded as sections of $\cL '(D_1+D_{a+1})$, but \emph{are} equal at $p_i$ when regarded as sections of $\cL'(D_1+D_a+D_{a+1})$, by applying \eqref{def V(xi)} for $(h,j)=(1,a),(a,a+1)$. Therefore, we must have $p_i\in D_a$, that is, $f\in N_a$. 

Because we furthermore have $f\in W_{a,a+1}$, we conclude that $f_a$ is zero at $p_i$ as a section of $\cL'(D_{a}+D_{a+1})$. Similarly, applying \eqref{def V(xi)} for $(h,j)=(1,a)$, we have that $f_a$ is zero at $p_i$ as a section of $\cL'(D_1+D_{a})$. Now, either $f_a$ is zero as a section of $\cL'(D_{a})$, in which case we are done, or $f\in N_1 \cap N_{a+1}$, in which case \eqref{def V(xi)} holds for $(h,j)=(1,a+1)$, contradicting the assumption at the beginning.

(2) is proven exactly as in Lemma \ref{lem:V_codim}, namely, $V_{a-1}\cap W_{a,a+1}$ is generically a sub-projective bundle of $\P$ of codimension $a$, and $V_{a-2}\cap W_a\cap N_a$ is generically a sub-projective bundle of the pullback of $\P$ over $N_a$ of codimension $a-1$. The details are omitted.

(3) is immediate, for example, from the fact that $N_a$ does not contain $V_{a-1}\cap W_{a,a+1}$.
\end{proof}

\begin{proof}[Proof of Proposition \ref{class of W(xi)}]
We proceed by induction on $a$. When $a=1$, this is Lemma \ref{div classes}. Suppose $a>1$. By the inductive hypothesis and Lemma \ref{recursion for [V]a}, we have
\begin{align*}
[V]_a=&[V]_{a-1}(\widetilde{\H}+\eta_a+\eta_{a+1})-[V]_{a-1}(\widetilde{\H}+\eta_a)\eta_a\\
=&\left( \widetilde{\H}^{a-1}+\sigma_1(\eta_1, \ldots, \eta_a) \widetilde{\H}^{a-2}+ \cdots+ \sigma_{a-1}(\eta_1,\ldots,\eta_{a}
) \right)(\widetilde{\H}+\eta_a+\eta_{a+1}) \\
&-\left( \widetilde{\H}^{a-2}+\sigma_1(\eta_1, \ldots, \eta_{a-1}) \widetilde{\H}^{a-3}+ \cdots+ \sigma_{a-2}(\eta_1,\ldots,\eta_{a-1}
) \right) (\widetilde{\H}+\eta_a)\eta_a \\
=& \widetilde{\H}^a+\sigma_1 (\eta_1, \ldots, \eta_{a+1})\widetilde{\H}^{a-1}+\cdots+ \sigma_{a}(\eta_1,\ldots,\eta_{a+1})
\end{align*}
where in the last equality we used the identities
\begin{align*}
(\eta_a+\eta_{a+1})\sigma_{i-1}(\eta_1,\ldots,\eta_a)+\sigma_i(\eta_1,&\ldots,  \eta_a)-\eta_a \sigma_{i-1}(\eta_1,\ldots,\eta_{a-1}) \\
&-\eta_a^2 \sigma_{i-2}(\eta_1,\ldots,\eta_{a-1})=\sigma_i(\eta_1,\ldots,\eta_{a+1})
\end{align*}
for $i=1,\ldots,a$.
\end{proof}

\subsubsection{Interlude on the permutohedral variety}\label{permutohedral_section}

For the proof of transversality, it will be convenient for formal reasons to pass from $X$ to a further blow-up $Y$, the \textit{permutohedral variety}, at higher-dimensional linear subspaces. We now describe this blow-up and its relevant properties.

 Fix a dimension $r\ge2$. Write $[r+1]=\{1,2,\ldots,r+1\}$. For a non-empty subset $S\subsetneq [r+1]$, let $\Lambda_S\subset\P^r$ denote the torus-invariant linear space of dimension $(\#S-1)$ given by the vanishing of the coordinates indexed by the complement $[r+1]\backslash S$.

 \begin{definition}
Let $\rho:Y\to X$ be the blow-up of $X$ (which is itself obtained by blowing up the torus-fixed points of $\P^r$) at the strict transforms of the $\binom{r+1}{2}$ torus-invariant lines of $\P^r$, followed by the strict transforms of the $\binom{r+1}{3}$ torus-invariant planes of $\P^r$, and so on, through the torus-invariant codimension 2 subspaces.
\end{definition}

 \begin{definition}
For $S\subset[r+1]$ of cardinality at most $r-1$, let $\E_S$ be the class of the strict transform of $\Lambda_S$. Let $\H^\vee,\E_S^\vee\in H_2(Y)$ be the basis of 1-cycles dual to the basis $\H,\E_S\in H^2(Y)$ of divisors.
 \end{definition}

 Note that $$K_Y=(r+1)\H-\sum_{\substack{S\subset[r+1] \\ 1\le\#S\le r-1}}(r-\#S)\E_S.$$

 A point $y\in Y$ can be expressed in coordinates as follows. First, take a point $y_0\in \P^r=\P(\C[[r+1]]),$ where $\C[[r+1]]$ denotes the vector space with basis given by the set $[r+1]=\{1,2,\ldots,r+1\}$. The point $y_0$ is the image of $y$ under the composite blow-up $Y\to\P^r$. Then, let $S^1\subsetneq [r+1]=:S^0$ be the subset of coordinates of $y_0$ which are equal to zero, corresponding to the minimal $T$-invariant subvariety of $\P^r$ in which $y_0$ lies. Then, let $y_1\in \P(\C[S^1])\cong\P^{\#S^1-1}$ be a point, representing a projectivized normal vector to $\Lambda_{(S^1)^c}$ at $y_0$. Define $S^2\subsetneq S^1$ analogously, as the set of coordinates of $y_1$ equal to zero, and continue until $y_k$ has all coordinates non-zero. Then, $y$ consists of the data of the points $(y_0,y_1,\ldots,y_k)$.

 Let $S\subset[r+1]$ be any non-empty subset. Then, we have a projection map $\rho_S:Y\to \P(\C[S])$ given by remembering the coordinates of $y=(y_0,y_1,\ldots,y_k)$ corresponding to $S$, at the unique point $y_{k'}$ for which $S\subseteq S^{k'}$ and the corresponding coordinates are not all zero. 

 Let $C$ be a smooth curve. A map $f:C\to Y$, which, upon post-composition with the blow-up $Y\to\P^r$, has image not contained in any torus-invariant subvariety, is given by the following data.

 \begin{itemize}
 \item a line bundle $\cL$ on $C$,
 \item for each $S\subset[r+1]$ with $1\le\#S\le r-1$, an effective divisor $D_S\subset C$, and
 \item for each $j\in[r+1]$, a non-zero section $$f_j\in H^0\left(C,\cL\left(-\sum_{j\notin S}D_S\right)\right),$$
 \end{itemize}
such that, when regarded as sections in $H^0(C,\cL)$, the $f_j$ have no common vanishing locus.
 
 The point $f(p)=(y_0,\ldots,y_k)$ may be computed as follows. First $y_0$ is the point $[f_0(p):\cdots:f_{r+1}(p)],$ where the $f_j$ are regarded as sections of $\cL$. Then, the subset $S^1\subset [r+1]$ is the set of indices $j$ for which $f_j(p)=0$. The point $y_1$ has coordinates given by the values of $f_j(p)$ after twisting $\cL$ down by the unique positive multiple of $p$ for which the $f_j(p)$ are well-defined and not all zero for $j \in S^1$. The rest of the $y_2,\ldots,y_k$ are then determined by further twists of $\cL$. 

\subsubsection{Transversality}\label{Proof of the Transversality Theorem}

In this section, we prove Proposition \ref{Understanding V(x1,...,xn)}. We continue to assume $\ell=r+1$. The main difficulty is to show that a point of $V(x_1,\ldots,x_n)$ represents an ``honest'' map in class $\beta$. More precisely, let $\P^\circ\subset\P$ be the open locus of $f=[f_1:\cdots:f_{r+1}]$ for which the $f_j$ share no common zeroes, when simultaneously viewed as sections of $\cL$. Then, we wish to show that $V(x_1,\ldots,x_n)\subset\P^\circ$. 

In order for our argument to work, we will first need the following lemma.

\begin{lemma}\label{coordinate_nonzero}
Suppose $f=[f_1:\cdots:f_{r+1}]$ is a point of $V(x_1,\ldots,x_n)$. Then, for each $j$, we have $f_j\neq0$ as a section of $\cL'(D_j)$ 
\end{lemma}

\begin{proof}
Suppose, without loss of generality, that $f_1=0$ and that $f_2\neq0$ as sections of $\cL'(D_1)$ and $\cL'(D_2)$, respectively (note that we cannot have all $f_j=0$). We may assume that none of the $x_i$ lie in the hyperplane in which $x_{i,1}=0$, so we must have that $f_2$ vanishes at \emph{all} $p_i$ as a section of $\cL'(D_1+D_2)$. In particular, we need $\deg(\cL'(D_1+D_2))\ge n$, that is, $$d-(k_3+\cdots+k_{r+1})\ge n,$$
which immediately contradicts the assumption that $n-d\ge g+1$.
\end{proof}

\begin{remark}\label{ineq_necessary}
We will later use the assumption $n-d\ge g+1$ again (see Lemma \ref{actual_dim_after_twisting}), but we point out that it is crucial for Lemma \ref{coordinate_nonzero} to hold, which in turn is necessary for the transversality. 

For example, when $g=0$ and $r\ge 3$, suppose that the class $\beta=d\H^\vee+k\E_1^\vee$ satisfies $0\le k\le d\le (r-1)k-r$. Then, we have $n\le d$, so one can take $f_2=\cdots=f_{r+1}=0$ and $f_1\in H^0(\P^1,\mathcal{O}(d))$ to vanish at all of the $p_1,\ldots,p_n$. This construction produces points of $V(x_1,\ldots,x_n)$ not lying in $\P^\circ$, and infinitely many such if $d>n$.

On the other hand, when $g=0$ and $r=2$, one can check that the assumption $n-d\ge 1$ can be dropped, also later in Lemma \ref{actual_dim_after_twisting}.
\end{remark}

We now turn to the main argument. The strategy is as follows: for any initial $f_\init\in V(x_1,\ldots,x_n)$, we describe an algorithm that essentially amounts to twisting down at base-points of $f_\init$ until we are able to define a map $f:C\to\P^r$. The map $f$ will also have various incidence conditions with respect to the torus-invariant loci of $\P^r$ and the $p_i,x_i$; a naive dimension count would immediately predict the nonexistence of such a map, unless $f\in\P^\circ$ to begin with. We show that this expectation holds by passing from $f$ to a map from $f:C\to Y$, where $Y$ is the permutohedral variety of the previous section, moving in a family of the expected dimension by Lemma \ref{basic_transversality}.

For the initial point $f_\init\in V(x_1,\ldots,x_n)$, denote the underlying line bundle by $\cL_\init$, the underlying divisors by $D_j$, and the underlying sections by
$$(f_j)_{\init}\in H^0\left(C,\cL_{\init}\left(-\sum_{v\neq j}D_v\right)\right).$$

Our procedure consists of a sequence of modifications to the following data.
\begin{itemize}
\item A line bundle $\cL$,
\item For every non-empty subset $S\subsetneq\{1,2,\ldots,r+1\}$, an effective divisor $D_S\subset C$, and
\item For all $j=1,2,\ldots,r+1$, a \emph{non-zero} (see Lemma \ref{coordinate_nonzero}) section $$f_j\in H^0\left(C,\cL\left(-\sum_{j\notin S}D_S\right)\right).$$
\end{itemize}

The notation $\cL,f_j$ contrasts with $\cL_\init,(f_j)_{\init}$, because $\cL,f_j$ will be modified throughout the algorithm, whereas $\cL_\init,(f_j)_{\init}$ are fixed throughout. Similarly, when $S=\{j\}$, we contrast the divisor $D_{\{j\}}$, which is modified throughout the algorithm, with $D_j$, which is fixed. The letter $f$ is only used at the end when we obtain a map $f:C\to\P^r$ after no base-points remain. 

The data is initialized as follows. 
\begin{itemize}
\item $\cL\leftarrow\cL_{\init}$,
\item $D_{\{j\}}\leftarrow D_j$ for $j=1,\ldots,r+1$, and $D_{S}\leftarrow 0$ if $\#S>1$, and 
\item $f_j\leftarrow(f_j)_{\init}$.
\end{itemize}

We will repeatedly consider the following conditions on $f_j$ given $p\in C$.
\begin{enumerate}
\item[(i)] $p\in D_{\{j\}}$
\item[(ii)] $f_j(p)=0$ as a section of $H^0\left(C,\cL\left(-\sum_{v\neq j}D_{\{v\}}\right)\right)$
\end{enumerate}

We now describe the algorithm. Write $C^\circ=C-\{p_1,\ldots,p_n\}$. The first three steps will be carried out independently for all points $p\in C^\circ$; we fix such a point $p$ throughout.

\begin{enumerate}[label=\underline{Step \arabic*}]

\item($p\in C^\circ$ satisfies (ii) for \emph{all} $j$)

Let $\alpha>0$ be the largest order to which $p$ vanishes at every $f_j$ (as a section of $\cL(-\sum_{v\neq j}D_{\{v\}})$) to order $\alpha$. Then, make the following modifications.
\begin{itemize}
\item $\cL$ is replaced by $\cL(-\alpha p)$,
\item each $D_S$ stays the same (in particular, all $D_S$ with $\#S>1$ remain empty), and
\item the new section $f_j\in H^0\left(C,\cL\left(-\sum_{v\neq j}D_{\{v\}}\right)\right)$ is taken to be the pre-image of the original section $f_j$ under the inclusion $$H^0\left(C,\cL(-\alpha p)\left(-\sum_{v\neq j}D_{\{v\}}\right)\right)\to H^0\left(C,\cL\left(-\sum_{v\neq j}D_{\{v\}}\right)\right).$$ (In other words, $f_j$ is ``twisted down by $\alpha$ at $p$.'')
\end{itemize}

After step 1, (ii) may still hold for some $j$, but will no longer hold for all $j$.

\item ($p\in C^\circ$ satisfies (i) and (ii) for \emph{some} $j$)

For any $p,j$ as above, let $\alpha_{p,j}>0$ be the maximal order to which both properties (i), (ii) hold, that is, $D_{\{j\}}$ contains $p$ with multiplicity $\alpha_{p,j}$ and $f_j$ vanishes at $p$ to order $\alpha_{p,j}$. Then, make the following modifications (independently for all $j$).
\begin{itemize}
\item $\cL$ is replaced by $\cL(-\alpha_{p,j}p)$,
\item $D_{\{j\}}$ is replaced by $D_{\{j\}}-\alpha_{p,j}p$, while all $D_S$ with $\#S>1$ remain empty, and
\item $f_j$ is twisted down by $\alpha_{p,j}$ at $p$.
\end{itemize}

After step 2, we have $\cL,D_S,f_j$ (still $D_S=0$ if $\#S>1$) with the property that, for any $p\in C^\circ$ and any $j\in[r+1]$, at most one of the properties (i), (ii) hold.

\item ($p\in C^\circ$ satisfies (i) for \emph{more than one} $j$)

 Let $S_1$ be the set of $j\in[r+1]$ for which (i) holds. Let $\alpha_j$ be the order to which $D_{\{j\}}$ contains $p$; in particular, $\alpha_j=0$ if $j\notin S_1$. Write $\alpha_{\max}$ for the largest of the $\alpha_j$, and write $\alpha_{\text{tot}}$ for the sum of all of the $\alpha_j$.

As a section of $\cL$, the order of vanishing of $f_j$ at $p$ is at least $\alpha_{\text{tot}}-\alpha_j$, and, after step 2, we have equality whenever $j\in S_1$. In particular, the common order of vanishing of the $f_j$ is equal to $\alpha_{\text{tot}}-\alpha_{\max}$, which is strictly positive exactly when $\#S_1\ge 2$. 

To remove the base-point at $p$, we wish therefore to twist down our sections $f_j\in H^0(\cL)$ by $\alpha_{\text{tot}}-\alpha_{\max}$ at $p$. After the twist, the new vanishing order of $f_j$ at $p$ (as a section of $H^0(\cL)$) will be $\alpha'_j:=\alpha_{\max}-\alpha_j$. We will keep track of this new vanishing condition via the divisors $D_S$.

This is achieved more precisely as follows. First, write $$0=\alpha'(1)<\cdots<\alpha'(t)$$ for the \emph{distinct} integers appearing among the $\alpha'_j$. (Note that $\alpha'(t)$ is equal to $\alpha_{\max}$ unless $S_1=[r+1]$.) Then, define the filtration
$$[r+1]=S^1\supsetneq S^2\supsetneq\cdots\supsetneq S^t$$
by $$S^m=\{j\in [r+1]\vert \alpha'_j\ge\alpha'(m)\}.$$
Note that $S^t=[r+1]-S_1$ if $S_1\subsetneq[r+1]$.

Finally, we make the following modifications to our data.
\begin{itemize}
\item $\cL$ is replaced by $\cL(-(\alpha_{\text{tot}}-\alpha_{\max})p)$,
\item each $D_{\{j\}}$ is (temporarily, possibly to be modified further in the bullet point below) replaced by $D_{\{j\}}-\alpha_jp$,
\item for $m=2,3,\ldots,t$, the divisor $D_{[r+1]\backslash S^m}$ is replaced by the divisor $D_{[r+1]\backslash S^m}+(\alpha'(m)-\alpha'(m-1))p$, and
\item $f_j$ stays the same. (Note that it is a section of the same line bundle as before; indeed, the multiplicity of $p$ in 
$$
 \sum_{S:j \notin S} D_S
$$
decreases by exactly $\alpha_{\text{tot}}-\alpha_{\max}$ in the previous two modifications.)
\end{itemize}

After step 3, no $p\in C^\circ$ is a base-point. Therefore, the $f_j$ (viewed as sections of $\cL$) define a map $f^\circ:C^\circ\to\P^r$, with the property that the divisor 
$$
D'_S=\sum_{S' \subseteq S}D_{S'}
$$
(when restricted to $C^\circ$) is constrained to map to the torus-invariant locus $\Lambda_S\subset \P^r$ defined earlier. There may be still additional vanishing; for example, the $f_j$ may have unexpected vanishing as sections of $\cL\left(\sum_{j\notin S}D_S\right)$, but we will not need to take this into account.

We now repeat the previous steps (with small modifications) on the points $p_i$. The steps will be carried out for each $p_i$ independently; we fix for the rest of this discussion an index $i$. We will furthermore assume for simplicity as in the previous section that $x_i=[1:\cdots:1]$. Care must be taken now to keep track of the effect of our twisting operations on the conditions $V(x_i)$.

\item($p_i$ satisfies (ii) for \emph{all} $j$)

This step is identical to step 1, with $p_i$ in place of $p$, and the $D_S$ with $\#S>1$ playing no role. After step 4, it will no longer be true that $p_i$ satisfies (ii) for all $j$.

\item
($p_i$ satisfies (i) for \emph{all} $j$)

This step is new. For a fixed $p_i$, $\alpha>0$ be the largest order to which $p_i$ is contained in all of the $D_j$ simultaneously. Then, make the following modifications.
\begin{itemize}
\item $\cL$ is replaced by $\cL((-\alpha r)p_i)$,
\item each $D_{\{j\}}$ is replaced by $D_{\{j\}}-\alpha p_i$, while all $D_S$ with $\#S>1$ remain unchanged, and
\item $f_j$ stays the same (note that it is a section of the same line bundle as before).
\end{itemize}

After step 5, it will no longer be true that $p_i$ satisfies (i) for all $j$. 

\begin{definition}
We say that $p_i$ is \emph{inactive} if either step 4 or 5 is run, that is, if $p_i$ initially satisfies (i) for all $j$ or (ii) for all $j$, and that $p_i$ is \emph{active} otherwise.
\end{definition}

If $p_i$ is inactive, then the equations $V(x_i)$ are automatically satisfied. Therefore, after the twisting of either step 4 or 5, $V(x_i)$ imposes no additional conditions on $f$.

\begin{lemma}\label{cond (i)=cond (ii) for active pi}
Suppose that $p_i$ is active. Then, condition (i) holds for $j$ if and only if (ii) does.
\end{lemma}

\begin{proof}
For convenience of notation, we argue here in terms of the initial divisors $D_j$, have not yet been modified in a neighborhood of $p_i$, and write $\cL'=\cL_\init(-D_1-\cdots-D_{r+1})$.

Suppose that $p_i\in D_j$. We may also assume that there exists a $j'$ that $p_i\notin D_{j'}$. Then, $f_j(p_i)=f_{j'}(p_i)=0$ in $\cL'(D_j+D_{j'})$, and because $p_i\notin D_{j'}$, we in fact have $f_j(p_i)=0$ in $\cL'(D_j)$.

Conversely, if $f_j(p_i)=0$ in $\cL'(D_j)$, then, for all $j'\neq j$, we have $f_{j'}(p_i)=0$ in $\cL'(D_{j}+D_{j'})$. If $p_i\notin D_{j}$, then in fact $f_{j'}(p_i)=0$ in $\cL'(D_{j'})$, for all $j'$ (including $j$), contradicting the assumption that $p_i$ is active. 
\end{proof}

If $p_i$ is active, write $S^i_\circ:=\{j\mid f_j\text{ satisfies both (i) and (ii) at $p_i$}\}$. Then, for any two $j,j'\notin S^i_\circ$, then $f_j(p_i)=f_{j'}(p_i)\neq0$ as sections of $\cL'(D_j+D_{j'})$. 

\begin{definition}
If $p_i$ is active, we say that $p_i$ is \emph{regular} if $S^i_\circ=\emptyset$, that is, $p_i$ is not a base-point of $f$. We say that $p_i$ is \emph{wild} otherwise.
\end{definition}

The final two steps are only needed for $p_i$ that are either inactive or wild.

\item ($p_i$ satisfies (i) and (ii) for \emph{some} $j$.)

Repeat step 2 with $p_i$ in place of $p$. The modifications are made exactly for $j\in S^i_\circ$. Let $S^i_{\circ,1},S^i_{\circ,2}\subset S^i_\circ$ denote the set of $j$ still satisfying conditions (i), (ii) after step 6.

\item ($p_i$ satisfies (i) for \emph{more than one} $j$)

Repeat step 3 with $p_i$ in place of $p$. Here, the subset $S^i_{\circ,1}$ plays the role of $S_1$ in step 3 above. Note that, by Lemma \ref{cond (i)=cond (ii) for active pi}, step $7$ is run only if step $6$ is.
\end{enumerate}

After step 7, $f$ is finally base-point free everywhere, and therefore defines a map $f:C\to\P^r$. We have already analyzed the incidence conditions imposed on $C^\circ$ (after step 3); we now do so at the $p_i$.

First, suppose that $p_i$ is inactive. As we have noted, the condition $V(x_i)$ no longer imposes constraints on the $f_j,D_S$ post-twisting. On the other hand, after steps 6 and 7, the divisors $D_S$ containing $p_i$ impose conditions on the intersection of $\text{im}(f)$ with the torus-invariant boundary of $\P^r$ at $p_i$ in exactly the same way as at the $p\in C^\circ$. 
Then, we summarize the situation.

For all non-empty $S\subsetneq[r+1]$, we have a divisor $D_S\subseteq C$ for which 
\begin{equation}\label{inactive_constraint}
f(D_S)\subseteq \Lambda_S.
\end{equation}
(Note that $D_S \subseteq D_{S}'$.) As we have already observed, we may be forgetting various other constraints. For example, the $f_j$ may have non-generic vanishing at various points of $C^\circ$ (including the $D_S$), the $D_S$ may contain copies of the same points with various multiplicites, and the $D_S$ may include the fixed inactive points $p_i$ in its support. The constraint \eqref{inactive_constraint} will turn out to be sufficient.

\begin{lemma}\label{lemma: active_constraint}
Suppose that $p_i$ is active. Then, 
\begin{equation}\label{active_constraint}
f(p_i)\in\langle x_i,\Lambda_{S^i_\circ}\rangle=:\Lambda_i,
\end{equation}
where $\langle -\rangle$ denotes linear span (so that $\Lambda_i$ has dimension $\# S^i_\circ$).
\end{lemma}
For regular $p_i$, note that \eqref{active_constraint} is simply the condition $f(p_i)=x_i$.
\begin{proof}
For any two $j,j'\notin S^i_\circ$, then $f_j(p_i)=f_{j'}(p_i)\neq0$ as sections of $$\cL\left(-\sum_{S\notin j,j'} D_S\right),$$ and this property remains unchanged by steps 6 and 7 (and upon injection into $\cL$).
\end{proof}

 We now come to the crux of the argument: combining the constraints \eqref{inactive_constraint} and \eqref{active_constraint} on $f$ enumerated above, we will pass from $f:C\to\P^r$ to a map $\bar{f}:C\to Y$, where $Y$ is the permutohedral variety of dimension $r$ defined in the previous section. Let $\bar{f}:C\to Y$ be unique lift of $f:C\to\P^r$ to $Y$; this $\bar{f}$ exists because the sections $f_j$ remain non-zero throughout the entire twisting process, so the image of $f$ is not contained in any torus-invariant subvariety.

 Write $\overline{\beta}:=\bar{f}_{*}[C]$, as well as $\overline{d}=\deg(\cL)$ (after twisting down all base-points) and $k_S=\deg(D_S)$ for the degree of the divisor $D_S$. The following lemma may be regarded as a numerical incarnation of \eqref{inactive_constraint}.

 \begin{lemma}\label{new_degree_bound}
We have 
\begin{equation*}
    \deg(\overline{\beta})\le (r+1)\overline{d} \; -\sum_{\substack{ S\subset[r+1] \\ 1\le\#S\le r-1}}(r-\#S)k_S,
\end{equation*}
where the degree is measured against $K_Y^\vee$.
 \end{lemma}

 \begin{proof}
Let $Y=Y_{r-1}\to Y_{r-2}\to\cdots \to Y_0=\P^r$ be the sequence of blow-ups to obtain $Y$, where the blow-up $\rho_s:Y_{s+1}\to Y_s$ is at the strict transforms of the $s$-dimensional torus-invariant subvarieties of $\P^r$. Note that the map $f:C\to\P^r$ has degree $\overline{d}$.

In the unique lift $f_s:C\to Y_s$ of $f$ to $Y_s$, for each $S$ with $\#S=s+1$, the divisor $D_S$ is constrained to map to strict transform of $\Lambda_S$ in $Y_s$. Therefore, the blow-up $\rho_s$ decreases the degree of $C$ by at least $(r-\#S)k_S=(r-s-1)k_S$ for each such $S$, that is,
$$(f_{s})_{*}[C]\cdot K_{Y_s}^\vee-(f_{s+1})_{*}[C]\cdot K_{Y_{s+1}}^\vee\ge(r-s-1)\sum_{\#S=s+1}k_S,$$
from which the lemma follows.
 \end{proof}

\begin{lemma}
At active $p_i$, we have
\begin{equation}\label{active_constraint_new}
f(p_i)\in\widetilde{\Lambda_i}.
\end{equation}
where $\widetilde{\Lambda_i}$ is the strict transform of $\Lambda_i$ in $Y$.
\end{lemma}
Note that if $p_i$ is regular, then we may identify $\widetilde{\Lambda_i}$ with $x_i$.

\begin{proof}
As $x_i\in X$ is general, we may assume it does not lie in any torus-invariant subvariety, and identify it both with a point of $Y$ and of $\P^r$. Then, the subvariety $\widetilde{\Lambda_i}\subset Y$ may be described in the language of \S\ref{permutohedral_section} as the fiber containing $x_i$ of the projection $\rho_{(S_\circ^i)^c}:Y\to\P(\C[(S_\circ^i)^c])$. That is, $\widetilde{\Lambda_i}$ is the set of points of $Y$ whose coordinates corresponding to the complement of $S_\circ^i$ are equal to those of $x_i$. When we take $x_i=[1:\cdots:1]$, this is to say that these coordinates are all non-zero and equal to each other.

The claim now follows from the fact that, by construction, the sections $f_j\in H^0(C,\cL)$ comprising $f:C\to\P^r$ have the same order of vanishing for $j\in (S_\circ^i)^c$. Therefore, twisting $\cL$ down by this order of vanishing yields a point $y_{k'}$ in the sequence $y=(y_0,\ldots,y_k)$ for which the corresponding coordinates are all non-zero and equal to each other, for the same reason as in Lemma \ref{lemma: active_constraint}. 
\end{proof}

If one or both of $S^i_{\circ,1},S^i_{\circ,2}$ are empty after step 6, then in fact, $f(p_i)$ is further constrained to lie in a proper subvariety of $\widetilde{\Lambda_i}$; we will, however, not need this. Note in particular that the expected number of conditions on $\bar{f}$ imposed by \eqref{active_constraint_new} is $r-\#S^i_\circ$.

\begin{definition}
Let $$\tau':\cM_{g,n}(Y,\bar{\beta})\to\cM_{g,n}\times \prod_{p_i\text{ active}}\P(\C[(S^i_\circ)^c])$$
be the map obtained in the first factor by remembering the domain curve, and in the second by evaluating at an active marked point $p_i$, and then projecting via $\rho_{(S^i_\circ)^c}$.
\end{definition}

Note that the projection $\rho_{(S^i_\circ)^c}$ is simply the blow-up map $Y\to\P^r$ when $p_i$ is regular.

\begin{lemma}\label{exp_dim_after_twisting}
The expected relative dimension of $\tau'$,
$$\mathrm{vdim}(\cM_{g,n}(Y,\bar{\beta}))-(3g-3+n)-\sum_{p_i \ \text{active} } (r-\mathrm{\#}S^i_0),$$
is non-positive, and is zero only if our starting point $f_{\text{init}}\in\P$ was base-point free, that is, in $\P^\circ$.
\end{lemma}

\begin{proof}
If $f_{\text{init}}\in\P$ to begin with, then $$\bar{\beta}\cdot K_Y^\vee\le \beta\cdot K_X^\vee,$$
 where $\beta$ is the original class $d\H^\vee+\sum_j k_j\E_j^\vee\in H_2(X)$, with equality whenever no two of the original sections $f_j\in H^0(C,\cL'(D_j))$ share a common vanishing point. Thus, the virtual dimension of $\cM_{g,n}(Y,\bar{\beta})$ is at most the dimension of $\cM_{g,n}\times X^{n}$ by \eqref{dim constraint}. In particular, the quantity in question is non-positive. 

In general, we will show that each of the steps of our twisting algorithm has the effect of decreasing the quantity
\begin{equation*}
\Omega:=\left[(r+1)\deg(\cL) \; -\sum_{\substack{S\subset[r+1] \\ 1\le\#S\le r-1}}(r-\#S)\deg(D_S)\right]-r(g-1)-\sum_{p_i\text{ active}}(r-\# S^i_\circ)
\end{equation*}

The quantities $\deg(\cL),\deg(D_S)$ are taken to be those that are changing throughout the twisting algorithm. As for the last term, we somewhat abusively set the values of $(r-\# S^i_\circ)$ to be equal to $r$ initially (as if all marked points $p_i$ are initially regular), so that the last summation is initially equal to $-rn$, and, by \eqref{dim constraint}, $\Omega$ is initially equal to zero. Then, if $p_i$ is inactive, the summand $(r-\# S^i_\circ)$ is removed after either step 4 or 5 (whichever is applied first). If $p_i$ is wild, then the value of $\#S^i_\circ$ is set to the correct one after step 6. If $p_i$ is regular, the summand $(r-\# S^i_\circ)$ remains equal to $r$ throughout.

By Lemma \ref{new_degree_bound}, the final value of $\Omega$ is an upper bound for the expected relative dimension of $\tau'$. Therefore, if $f_{\text{init}}\notin\P^\circ$, that is, at least one step of the algorithm is run, then the expected relative dimension of $\tau'$ will be strictly negative, as needed. 

\begin{enumerate}[label=\underline{Step \arabic*}]
\item We decrease $\deg(\cL)$ by $\alpha>0$ and make no other changes, which decreases $\Omega$ by $(r+1)\alpha$.

\item For each $j$, we decrease $\deg(\cL)$ and $\deg(D_{\{j\}})$ each by $\alpha_{p,j}>0$, which decreases $\Omega$ by $2\alpha_{p,j}$.

\item First, decreasing $\deg(\cL)$ by $\alpha_{\text{tot}}-\alpha_{\max}$ decreases $\Omega$ by $(r+1)(\alpha_{\text{tot}}-\alpha_{\max})$, and replacing each $D_{\{j\}}$ by $D_{\{j\}}-\alpha_j p$ increases it by $(r-1)\alpha_{\text{tot}}$. Finally, the modification of $D_{[r+1]\backslash S^m}$ decreases $\Omega$ by $(\#S^m-1)(\alpha'(m)-\alpha'(m-1))$. We thus need to show that
$$-2\alpha_{\text{tot}}+(r+1)\alpha_{\max}-\sum_{m=2}^{t}(\#S^m-1)(\alpha'(m)-\alpha'(m-1))<0.$$ 

Let $$
\alpha(1)>\alpha(2)>\cdots>\alpha(t)
$$ denote the distinct integers among the $\alpha_j$; note that $\alpha(t)=0$ if and only if $S_1\subsetneq[r+1]$. Let $r_1,\ldots,r_t$ denote the number of times $\alpha(1),\ldots,\alpha(t)$ appear among the $\alpha_j$, so that $r+1=r_1+\cdots+r_t$. We now have
\begin{align*}
&-2\alpha_{\text{tot}}+(r+1)\alpha_{\max}-\sum_{m=2}^{t}(\#S^m-1)(\alpha'(m)-\alpha'(m-1))\\
=&-2\sum_{m=1}^{t}\alpha(m) r_m+(r+1)\alpha(1)-\sum_{m=2}^{t}(r_m+\cdots+r_t-1)(\alpha(m-1)-\alpha(m))\\
=&-\sum_{m=1}^{t}r_m\alpha(m)+\alpha(1)-\alpha(t)<0.
\end{align*}

\item As in step 1, we decrease $\deg(\cL)$ by $\alpha>0$ and make no other changes, which decreases $\Omega$ by $(r+1)\alpha$. On the other hand, the term $(r-\#S^i_\circ)$ goes away. Thus, the expected relative dimension of $\tau'$ goes down by $(r+1)\alpha-r>0$.

\item We decrease all of the $\deg(D_{\{j\}})$ by $\alpha>0$ and $\deg(\cL)$ by $r\alpha$, decreasing $\Omega$ by $(r+1)\alpha$, and possibly remove the term $(r-\#S^i_\circ)$ (if this was not done in the previous step), so $\Omega$ goes down by at least $(r+1)\alpha-r>0$.
\end{enumerate}
We examine the effect of the last two steps depending on the type of $p_i$.
\begin{itemize}
\item if $p_i$ is inactive, then applications of steps 6 and 7 only decrease $\Omega$ further, by steps 2 and 3. 

\item if $p_i$ is regular, then no changes are made.

\item if $p_i$ is wild, then step 6 decreases $\deg(\cL)$ and $\deg(D_{\{j\}})$ together by at least 1 for each $j\in S^i_\circ$, decreasing $\Omega$ by at least $2\cdot\# S^i_\circ$. Furthermore, the $-(r-\#S^i_\circ)$ is updated to the correct value (from $-r$ originally) in this step. In total, the value of $\Omega$ goes down at least by $$2\cdot\# S^i_\circ+(r-\#S^i_\circ)-r>0$$
in step 6. In step 7, $\Omega$ can again only decrease, by step 3. (Note that step $7$ does not affect the contribution of $-(r-\#S^i_\circ)$.)
\end{itemize}
\end{proof}

\begin{lemma}\label{actual_dim_after_twisting}
Every irreducible component of $\cM_{g,n}(Y,\overline{\beta})$ dominating $\cM_{g,n} \times \prod_{p_i \ \text{regular}} X$ (where we may equivalently replace $X$ in the last factor with $Y$ or $\P^r$) has expected dimension.
\end{lemma}

\begin{proof}

We claim that there are at least $g+1$ regular points among the $p_i$. Indeed, every point $p_i$ which is a base-point for our original $f\in \P$ contributes at least 1 to the degree $d$ of our original line bundle $\cL$, but by assumption, we have $d\le n- (g+1)$, so at least $ g+1$ regular points remain after our twisting algorithm.

Therefore, the conclusion follows from Lemma \ref{basic_transversality}. 
\end{proof}

\begin{proof}[Proof of Proposition \ref{Understanding V(x1,...,xn)}]
Lemmas \ref{exp_dim_after_twisting} and \ref{actual_dim_after_twisting} at long last imply that $V(x_1,\ldots,x_n)$ is contained in $\P^\circ$. Indeed, if this is not the case for some $f\in V(x_1,\ldots,x_n)$, then applying the twisting algorithm to $f$ yields a point on $\cM_{g,n}(Y,\bar{\beta})$ moving in a family dominating the target of $\tau'$, but this is a contradiction due to dimension reasons.

Similarly, for dimension reasons, we see that $f\in V(x_1,\ldots,x_n)$ must define a (base-point free) map $f:C\to X$ of curve class \emph{exactly} $\beta$, as unexpected vanishing of the $f_j$ can only decrease the degree of $\beta$. 

It remains to check that any $f\in V(x_1,\ldots,x_n)$ has no non-trivial tangent vectors. It is elementary to check that such a tangent vector would give rise to a non-trivial relative tangent vector of the map $\tau:\cM_{g,n}(X,\beta)\to\cM_{g,n}\times X^n$ over a general point, which contradicts Lemma \ref{basic_transversality}.
\end{proof}

\subsection{Integral formula on $S$}

In this section, we push forward the formula of Theorem \ref{thm:integral formula} to obtain Theorem \ref{thm:grr}. Up to setting $k_i=0$ for $i>\ell$, we can assume that $\ell=r+1$.

We start with a change of variables. Define
$$
\widetilde{\cE}:= \cE \otimes \mathcal{O}(N) = \bigoplus_{i=1}^{r+1} \nu_*(\mathcal{P}(-\mathcal{D}+N +\mathcal{D}_i))=\bigoplus_{i=1}^{r+1} \nu_*(\mathcal{P}(\Delta_i)),
$$
where we have written
$$
N=\sum_{i=1}^{r+1} N_i, \ \mathcal{D}= \sum_{i=1}^{r+1} \mathcal{D}_i \text{ and } \Delta_i=-\mathcal{D}+N +\mathcal{D}_i
\text{ for } i=1,\ldots,r+1.
$$
Using $\widetilde{\cE}$, we can rewrite the formula \eqref{eqn:integral formula} as
$$
    \mathsf{Tev}_{g,n,\beta}^X= \int_{\P(\widetilde{\cE})} (\widetilde{\H}^r+\sigma_1 \widetilde{\H}^{r-1}+\cdots.+\sigma_r)^n
$$
where, by an abuse of notation,  $\widetilde{\H}=c_1(\mathcal{O}_{\P(\widetilde{\cE})}(1))$. Writing 
\begin{align*}
    (\widetilde{\H}^r+\sigma_1 \widetilde{\H}^{r-1}+\cdots.+\sigma_r)^n&= \left( \frac{\prod_{i=1}^{r+1}(\widetilde{\H}+\eta_i)-\sigma_{r+1}}{\widetilde{\H}} \right)^n \\
    &= \sum_{m=0}^n \binom{n}{m} (-1)^m \frac{\prod_{i=1}^{r+1}(\widetilde{\H}+\eta_i)^{n-m} \sigma_{r+1}^m}{\widetilde{\H}^n} \\
    &= \sum_{m=0}^n \binom{n}{m} (-1)^m \frac{\prod_{i=1}^{r+1}(\sum_{a_i=0}^{n-m} \binom{n-m}{a_i} \widetilde{\H}^{n-m-a_i}\eta_i^{a_i+m}) }{\widetilde{\H}^n},
\end{align*}
we obtain
\begin{equation} \label{eqn:integral to compute before GRR}
\begin{aligned}
    \mathsf{Tev}_{g,n,\beta}^X&=\sum_{m=0}^n \binom{n}{m} (-1)^m \int_{\P(\widetilde{\cE})}\frac{\prod_{i=1}^{r+1}(\sum_{a_i=0}^{n-m} \binom{n-m}{a_i} \widetilde{\H}^{n-m-a_i}\eta_i^{a_i+m}) }{\widetilde{\H}^n} \\
    &=\sum_{m=0}^{\min(n,k_1,\ldots,k_{r+1})} \binom{n}{m}(-1)^m \int_{S} \left( \prod_{i=1}^{r+1}(1+\eta_i)^{n-m} \eta_i^m \right) c(\widetilde{\cE})^{-1},
\end{aligned}
\end{equation}
where in the last equality we have pushed forward to $S$. It remains to compute $c(\widetilde{\cE})$; this will be an application of Grothendieck-Riemann-Roch.

\subsubsection{Cohomology of $\J(C) \times \Sym^{k_1}(C) \times \cdots \times \Sym^{k_{r+1}}(C)$}\label{coh_of_S}

We now fix the notation necessary to compute the integral \eqref{eqn:integral to compute before GRR}.

Let $e_1,\ldots,e_{2g}$
be a symplectic basis of $H^1(C,\Z)$, and denote by $e_1',\ldots,e_{2g}'$ the corresponding classes in $H^1(\J(C),\Z)$ under the natural isomorphism
$$
H^1(\J(C),\Z) \rightarrow H^1(C,\Z).
$$
Let also 
$$
\Theta= \sum_{\alpha=1}^g e_\alpha' e_{\alpha+g}'
$$
be the class of the theta divisor on $\J(C)$ and set 
$$
\gamma=- \left( \sum_{\alpha=1}^g e_\alpha' e_{\alpha+g}-e_{\alpha+g}'e_{\alpha} \right) \in H^1(\J(C),\Z) \otimes H^1(C,\Z) .
$$
By \cite[Chapter VIII]{ACGH}, we have
\begin{equation}\label{Chern class of Poicare line bundle}
\mathrm{ch}(\mathcal{P})=1+d \p +\gamma-\Theta \p
\end{equation}
where $\p \in H^2(C,\Z)$ is the point class.

Next, we recall MacDonald's description \cite{MacDonald} of the cohomology ring of $\Sym^{k}(C)$. For $i=1,\ldots,r+1$, define classes $\zeta_{i,1},\ldots,\zeta_{i,2g} \in H^1(\Sym^{k_i}(C),\Z)$ for $\eta_i \in H^2(\Sym^{k_i}(C),\Z)$ as the Künneth components of the universal divisor $\mathcal{D}_i$:
$$
\mathcal{D}_i= \eta_i +\sum_{\alpha=1}^g (\zeta_{i,\alpha}e_{\alpha+g}-\zeta_{i,\alpha+g}e_\alpha)+k_i \p.
$$
Note that $\eta_i=[N_i]$ as before. These generate the ring $H^*(\Sym^{k_i}(C),\Z)$, and setting $\tau_{i,\alpha}=\zeta_{i,\alpha} \zeta_{i,\alpha+g}$, for every multiindex $I$ without repetitions,  we have 
$$
\int_{\Sym^{k_i}(C)} \eta_i^{k_i-|I|} \tau_{i,I}= 1
$$
where $\tau_{i,I}=\prod_{\alpha \in I} \tau_{i,\alpha}$.

Set also 
\begin{align*}
&\overline{y}_i=\sum_{\alpha=1}^g \sum_{\substack{1 \leq j \leq r+1 \\ j \neq i}} (\zeta_{j,\alpha}e_{\alpha+g}-\zeta_{j,\alpha+g} e_\alpha), \\
& \overline{\tau}_i= \sum_{\alpha=1}^g \sum_{\substack{1 \leq j_1,j_2 \leq r+1 \\ j_1,j_2 \neq i}} \zeta_{j_1,\alpha} \zeta_{j_2,\alpha+g}, \\
& \overline{x}_i=\sum_{\alpha=1}^g \sum_{\substack{1 \leq j \leq r+1 \\ j \neq i}} e_\alpha' \zeta_{j,\alpha+g}-e_{\alpha+g}' \zeta_{j,\alpha}, \\
\text{and } &\overline{k}_i= \sum_{\substack{1 \leq j \leq r+1 \\ j \neq i}}k_j,
\end{align*}
and write
$$
\Delta_i= \eta_i-\overline{y}_i-\overline{k}_i \p.
$$
The following lemma will be used in Section \S\ref{Section: computation c(tilde(E))}.

\begin{lemma}\label{lemma:cohomological identities}
    The following identities hold.
\begin{enumerate}
    \item[(i)] 
    $\Delta_i^h= \eta_i^h-\eta_i^{h-1}(h \overline{y}_i+h \overline{k}_i \p)+ \eta_i^{h-2}(-h(h-1) \overline{\tau}_i \p) $
    \item[(ii)]
    $\gamma \overline{y}_i= \overline{x}_i \p$ 
    \end{enumerate}
\end{lemma}
\begin{proof}
 Identity (i) is proved by induction on $h$ using $\overline{y}_i^2=-2 \overline{\tau}_i \p$. Identity (ii) is straightforward.
\end{proof}

\subsubsection{Computation of $c(\widetilde{\cE})$} \label{Section: computation c(tilde(E))}

We now express $c(\widetilde{\cE})$ in terms of the classes defined above. Fix an index $i$ with $1\le i\le r+1$.
\begin{proposition}
 We have
 $$
 \mathrm{ch}( \nu_*(\mathcal{P}(\Delta_i))= \mathrm{exp}(\eta_i) (1-g+d- \overline{k}_i-\overline{\tau}_i- \Theta+\overline{x}_i).
 $$
\end{proposition}
\begin{proof}
By Grothendieck-Riemann-Roch and \eqref{Chern class of Poicare line bundle}, we have
\begin{align*}
\mathrm{ch}(\nu_*(\mathcal{P}(\Delta_i))&=\nu_*(\mathrm{ch}(\mathcal{P}(\Delta_i))\cdot \mathsf{Td}_C)\\
&=\nu_*((1+d \p+\gamma-\p \Theta) \cdot\mathrm{exp}(\Delta_i)\cdot(1+(1-g) \p).
\end{align*}
We compute
\begin{align*}
\nu_*( \mathrm{exp}(\Delta_i))&= \nu_*\left( \sum_{h=0}^\infty \frac{\Delta_i^h}{h!}\right) \\
&=\nu_*\left( \sum_{h=0}^\infty \frac{-h \overline{k}_i}{h!} \eta_i^{h-1} \p - \frac{h(h-1)}{h!} \eta_i^{h-2}\overline{\tau}_i \p\right) \\
&= \mathrm{exp}(\eta_i)(-\overline{k}_i-\overline{\tau}_i)
\end{align*}
where in the second equality we used identity (i) in Lemma \ref{lemma:cohomological identities}. Similarly, we have
\begin{align*}
    &\nu_*((d-\Theta) \p \ \mathrm{exp}(\Delta_i))= \mathrm{exp}(\eta_i)(d-\Theta), \\
    &\nu_*(-\gamma \ \mathrm{exp}(\Delta_i))= \mathrm{exp}(\eta_i) \overline{x}_i, \\
    & \nu_*((1-g) \p \ \mathrm{exp}(\Delta_i) \mathrm{ch}(\mathcal{P}))=(1-g) \mathrm{exp}(\eta_i).
\end{align*}
Here, in the second equality we have used identity (ii) in Lemma \ref{lemma:cohomological identities}. After summing all the obtained contributions, we obtain the claim. 
\end{proof}

\begin{corollary}\label{chern_formula}
$$
c(\widetilde{\cE})= \prod_{i=1}^{r+1} (1+\eta_i)^{1-g+d-\overline{k}_i}\cdot \mathrm{exp}\left( \frac{- \overline{\tau}_i- \Theta+ \overline{x}_i}{1+\eta_i} \right)
$$
\end{corollary}
\begin{proof}
    This follows immediately from the formula 
    $$
    c(V)= \mathrm{exp}\bigg( \sum_{h=1}^\infty (-1)^{h-1}(h-1)!\cdot\mathrm{ch}_h(V) \bigg)
    $$
    valid for any vector bundle $V$ on any scheme.
\end{proof}

Substituting \ref{chern_formula} into \eqref{eqn:integral to compute before GRR} completes the proof of Theorem \ref{thm:grr}, namely that
 $$\mathsf{Tev}_{g,n,\beta}^X=\sum_{m=0}^{\min(n,k_1,\ldots,k_{r+1})} \binom{n}{m}(-1)^m \int_{S} \prod_{i=1}^{r+1}(1+\eta_i)^{n-m-1+g-d+\overline{k}_i} \eta_i^m   \cdot\mathrm{exp}\left( \frac{ \overline{\tau}_i+ \Theta- \overline{x}_i}{1+\eta_i} \right).
$$

\subsection{Specialization to genus 0}\label{Computations in genus 0}

We now specialize Theorem \ref{thm:grr} to $g=0$ to prove Theorem \ref{explicit formulas in genus 0}.

We continue to assume that $\ell=r+1$. Identify
$$
\Sym^{k_i}(\P^1)=\P^{k_i} \text{ for } i=1,\ldots,r+1.
$$
Then, the class $\eta_i$ is the first Chern class of $\mathcal{O}_{\P^{k_i}}(1)$ and Theorem \ref{thm:grr} gives:
\begin{align*}
   \mathsf{Tev}_{0,n,\beta}^X&=\sum_{m=0}^{\min(n,k_1,\ldots,k_{r+1})} \binom{n}{m}(-1)^m \int_{\prod_{i=1}^r \P^{k_i}} \prod_{i=1}^{r+1}(1+\eta_i)^{n-m-1-d+\overline{k}_i} \eta_i^m \\
   &= \sum_{m=0}^{\min(n,k_1,\ldots,k_{r+1})} \binom{n}{m}(-1)^m  \prod_{i=1}^{r+1}\mathrm{Coeff}\left( (1+\eta_i)^{n-m-1-d+\overline{k}_i};\eta_i^{k_i-m}\right)\\
   &=\sum_{m=0}^{\min(n,k_1,\ldots,k_{r+1})} \binom{n}{m}(-1)^m  \prod_{i=1}^{r+1}\binom{n-m-d-1+\overline{k}_i}{k_i-m}.
\end{align*}

This concludes the proof of Theorem \ref{explicit formulas in genus 0}.

\subsection{Specialization to $X=\Bl_q(\P^r)$}\label{Computations for blow-ups at 1 point}

 Theorem \ref{thm:grr} also admits a reasonably elegant specialization to $X=\Bl_q(\P^r)$. We take now $k_2=\ldots=k_{r+1}=0$ and write $k=k_1$, $\eta=\eta_1$, and $\E=\E_1$, so that $\tau=\overline{\tau}_2=\cdots=\overline{\tau}_{r+1}$ and $x=\overline{x}_2=\cdots=\overline{x}_{r+1}$.

We have 
\begin{align*}
\mathsf{Tev}_{g,n,d\H^\vee+k \E^\vee}^X&= \int_{\J(C) \times \Sym^{k}(C)} (1+\eta)^{n-1+g-d} \cdot \mathrm{exp} \left(\frac{\Theta}{1+\eta}\right) \cdot \mathrm{exp}(\tau+\Theta-x)\\
&=\int_{\J(C) \times \Sym^{k}(C)} A(\eta) \cdot\mathrm{exp}(\tau B(\eta)+\Theta C(\eta)+xD(\eta)),
\end{align*}
where 
\begin{equation*}
\begin{cases}
A(\eta)&=(1+\eta)^{n-1+g-d},
\\B(\eta)&= -r,
\\C(\eta)&= \frac{r\eta+(r+1)}{1+\eta},
\\D(\eta)&= r.
\end{cases}
\end{equation*}

\begin{lemma}\label{ABCDintegral}
Let $A(\eta),B(\eta),C(\eta)$ and $D(\eta)$ be polynomials in $\eta$. Then,
\begin{align*}
\int_{\J(C) \times \Sym^{k}(C)} A(\eta) \cdot\mathrm{exp}(\tau B(\eta)+ & \Theta C(\eta)+ xD(\eta)) \\
&= \mathrm{Coeff}\left( A(\eta) ( C(\eta)(1+\eta B(\eta))-\eta D(\eta)^2 )^g;\eta^k \right)
\end{align*}
\end{lemma}

\begin{proof}
Notice that for $i+j=g$, we have 
$$
\Theta^i x^{2j}=\sum_{\substack{I=(i_1,\ldots,i_j)\\ 1 \leq i_1<\ldots<i_j\leq g}} i!(-1)^j(2j)! e_1'\cdots e_{2g}' \tau_{I},
$$
where $\tau_I= \prod_{i \in I} \tau_{i}$ and $\tau_{\alpha}=\zeta_{1,\alpha}\zeta_{1,\alpha+g}$ for $\alpha=1,\ldots,g$. Therefore, we can expand
\begin{align*}
&\int_{\J(C) \times  \Sym^{k}(C)} A(\eta) \cdot\mathrm{exp}(\tau B(\eta)+\Theta C(\eta)+xD(\eta))\\&=\sum_{m=0}^g (-1)^{g-m} \sum_{\substack{I=(i_1,\ldots,i_{g-m})\\ 1 \leq i_1<\cdots<i_{g-m}\leq g}} \int_{\Sym^k(C)}A(\eta)C(\eta)^mD(\eta)^{2(g-m)} \cdot \mathrm{exp}(\tau B(\eta)) \tau_I.
\end{align*}
Call
\begin{equation}\label{tilde(A)}
\widetilde{A}(\eta)=A(\eta)C(\eta)^mD(\eta)^{2(g-m)}.
\end{equation}
We have 
\begin{align*}
\int_{\Sym^k(C)}\widetilde{A}(\eta) \cdot\mathrm{exp}(\tau B(\eta)) \tau_I&= \sum_{h=0}^\infty \int_{\Sym^k(C)} \widetilde{A}(\eta) \frac{\tau_I \tau^h B(\eta)^h }{h!}\\
&=\sum_{h=0}^\infty \int_{\Sym^k(C)} \widetilde{A}(\eta) \left( \sum_{\substack{J: |J|=h \\ J \cap I =\emptyset}}\tau_{I}\tau_J \right)B(\eta)^h \\
&=\sum_{h=0}^{g-|I|} \binom{g-|I|}{|J|} \cdot\mathrm{Res}_{\eta=0} \left\{ \frac{\widetilde{A}(\eta) B(\eta)^h}{\eta^{k-h-|I|+1}} d\eta \right\} \\
&=\mathrm{Res}_{\eta=0} \left\{ \frac{\widetilde{A}(\eta)}{\eta^{k-|I|+1}} (1+\eta B(\eta))^{g-|I|} d\eta \right\}.
\end{align*}
The claim follows then by substituting \eqref{tilde(A)} and using the combinatorial identity
$$
\sum_{m=0}^g \binom{g}{m} (C(\eta)(1+\eta B(\eta))^m(-D(\eta)^2\eta)^{g-m}= \left( C(\eta)(1+\eta B(\eta)) -\eta D(\eta)^2 \right)^g
$$
\end{proof}
From Lemma \ref{ABCDintegral}, we obtain the residue formula
\begin{align*}
\mathsf{Tev}_{g,n,d\H^\vee+k \E^\vee}^X&= \mathrm{Coeff} \left(  (1+\eta)^{n-1-d} (2r \eta +(r+1))^g ; \eta^k \right), 
\end{align*}
and the right hand side can be expanded as 
\begin{align*}
 &\sum_{m=0}^m \binom{g}{m} (2r)^m (r+1)^{g-m} \binom{n-1-d}{k-m}\\
 = &\sum_{m=0}^m \binom{g}{m} (2r)^m \sum_{h=0}^{g-m} \binom{g-m}{h}(1-r)^{g-m-h}(2r)^{h} \binom{n-1-d}{k-m} \\
 =&\sum_{a=0}^g \binom{g}{a}(2r)^a (1-r)^{g-a} \sum_{m=0}^a \binom{a}{m} \binom{n-1-d}{k-m} \\
 =&\sum_{a=0}^g \binom{g}{a} (2r)^a (1-r)^{g-a} \binom{g}{a} \binom{n-1-d+a}{k},
\end{align*}
where in the second equality we have set $a=m+h$ and in the third equality we have applied the Vandermonde identity
$$
{\displaystyle {m+n \choose r}=\sum _{k=0}^{r}{m \choose k}{n \choose r-k}}
 \text{ for all }r,n,m \in \Z_{\geq 0}.
 $$
Finally, replacing $a$ with $g-m$, we obtain Theorem \ref{explicit formulas for 1 pt blow-up}.

\section{Virtual counts for $X=\Bl_q(\P^r)$}\label{the virtual count}

In this section, we establish the virtual count of Theorem \ref{virtual formulas for 1 pt blow-up}. 

\subsection{Preliminaries on $QH^*(X)$}
 
The computation of the invariants $\vTev^X_{g,n,\beta}$ is reduced to a calculation in the quantum cohomology ring $QH^*(X,\Q)$ in \cite{BP}. See \cite{FP} for an introduction. We now take $X=\Bl_q(\P^r)$.

Because $X$ is Fano, by \cite[Proposition 2.2]{ST}, we have 
$$
QH^*(X, \Q) \cong \Q[\mathsf{C}] \otimes_\Q H^*(X,\Q)
$$
as $\Q[\mathsf{C}]$-modules. Here, $\mathsf{C}\subseteq H_2(X,\Z)$ is the cone of effective curves in $X$, and by \cite[Proposition 4.1]{CLO}, we have 
$$
\mathsf{C}= \Z_{\geq 0}(\H^\vee+\E^\vee) \oplus \Z_{\geq 0} (-\E^\vee).
$$
We will denote by $\star$ the product in $QH^*(X,\Q)$.

Let $\Sigma \subseteq \mathbb{R}^{r}$ be toric fan of $X$. Its generators are $v_j=e_{j+1}$ for $j=0,1,\ldots,r-1$, along with $v_r=-(e_1+\cdots+e_r)$ and $v=-v_r$.

The sets 
$$
S=\{v_0,\ldots,v_{r-1} \} \text{ and } T=\{v,v_{r-1}\}
$$
are the only two primitive sets in the sense of \cite[Definition 1.1]{Kresch} and, using \cite[Theorem 1.2]{Kresch}, from these two sets we respectively obtain the following two relations in $QH^*(X,\Q)$.
\begin{equation}\label{relations}
\begin{aligned}
(\H-\E)^{\star r}&=q^{-\E^\vee}\E; \\
\H \star \E&=q^{\H^\vee+\E^\vee}.
\end{aligned}
\end{equation}

\begin{lemma}\label{lemma: basis quantum cohomology}
The set
$$
1,(\H-\E),\ldots,(\H-\E)^{\star r-1}, \E \star (\H-\E),\ldots.,\E \star (\H-\E)^{r-1}
$$
is a basis of $QH^*(X,\Q)$ as $\Q[\mathsf{C}]$-module.
\end{lemma}
\begin{proof}
From the second relation of \eqref{relations}, we have
$$
\E^{\star 2}=-(\H-\E) \star \E +q^{\H^\vee+\E^\vee},
$$ so the above collection spans $QH^*(X,\Q)$. The fact that they are linearly independent follows from \cite[Theorem 1.2]{Kresch}.
\end{proof}
Consider the homogeneous basis 
\begin{equation}\label{eqn: basis cohomology}
\H^r,\H^{r-1},\ldots,1,\E^{r-1},\ldots,\E
\end{equation}
of $H^*(X,\Q)$. We now express each of these classes in $QH^*(X,\Q)$ as linear combinations of the basis elements of Lemma \ref{lemma: basis quantum cohomology} using \cite[Theorem 1.6]{Kresch}. In order for \cite[Theorem 1.6]{Kresch} to apply, we need the following result.

\begin{lemma}
Every toric subvariety of $X$ is Fano. Furthermore, for every map $\pi:X \to X'$, where $X'$ is toric and $\pi$ is the blowuup of an irreducible toric subvariety, we have that $X'$ is Fano.
\end{lemma}
\begin{proof}
Immediate consequence of \cite[Theorem 3.9]{Kresch}.
\end{proof}

\begin{lemma}\label{lemma: cohomology classes in quantum cohomology}
We have
\begin{align*}
\H^i&=\H \star (\H- \E)^{\star (i-1)} \text{ for } i=1,\ldots,r, \text{ and} \\
\E^i&=(-1)^{i-1} \E \star (\H-\E)^{\star (i-1)} \text{ for } i=1,\ldots,r-1
\end{align*}
in $QH^*(X,\Q)$.
\end{lemma}
\begin{proof}
We start with $\H^i$. Let $\alpha_i$ be the cone generated by $v_r,\ldots,v_{i+1}$ for $i=1,\ldots,r$. We have that $\alpha_i$ corresponds to a toric subvariety of class $\H^{r-i}$. Then, \cite[Theorem 1.6]{Kresch} and the fact the only exceptional set special for $\alpha_i$ is the empty set yields the desired equality.

The proof for $\E^i$ is similar. This time, $\alpha_i$ is the cone generated by $v$ and  $v_{r-2},\ldots,v_{r-i}$ for $i=1,\ldots,r-1$ (when $i=1$, we just have $\alpha_1=\langle v \rangle$), so $\alpha_i$ corresponds to $(-1)^{i+1}\E^i$. 
\end{proof}

\subsection{The quantum Euler class of $X$}

By definition, the \textbf{quantum Euler class} of $X$ is the image of the Künneth decomposition of the diagonal in $X \times X$ under the natural product map
$$
H^*(X,\Q) \otimes H^*(X,\Q) \xrightarrow{\star} QH^*(X,\Q).
$$
This is a canonically defined element of $QH^*(X,\Q)$, first introduced by Abrams in \cite{Abr} (see also \cite{BP,Cela}).

The computation of the virtual invariants $\vTev^X_{g,n,d \H^\vee +k \E^\vee}$ is related to $\Delta$ in \cite[Theorem 1.3]{BP} by the formula:

$$
\vTev^X_{g,n,d \H^\vee +k \E^\vee}=\mathrm{Coeff}(\mathsf{P}^{\star n}\star \Delta^{\star g},\mathsf{P} q^{d \H^\vee +k \E^\vee}).
$$
where $\mathsf{P}$ is the point class in $X$.
\begin{theorem}\label{Quantum Euler Class blow-up 1 pt}
We have
$$
\Delta=(2r)\mathsf{P}-(r-1)q^{-\E^\vee}\E.
$$
in $QH^*(X,\Q)$.
\end{theorem}

\begin{proof}
Let 
$$
1,\H,\ldots,\H^r,(-1)^{r-1}\E,\ldots.,(-1)^{r-1}\E^{r-1}.
$$
be the dual basis to \eqref{eqn: basis cohomology} with respect to the Poincar\'{e} pairing. Then 
\begin{align*}
\Delta=& \mathsf{P} \star 1 +\cdots+\H^{r-i} \star \H^i +\cdots+1 \star \mathsf{P} +(-1)^{r-1} ( \E \star \E^{r-1}+\cdots+\E^{r-1} \star \E)\\
=& 2 \mathsf{P}+(r-1) \H^{\star 2} \star (\H-\E)^{\star (r-2)}- (r-1)\E^{\star 2} \star (\H-\E)^{\star (r-2)} \\
=& 2 \mathsf{P}+(r-1) \left( (\H-\E)^{r}+\E^2 \star (\H-\E)^{r-2}+2 \E \star (\H- \E)^{\star (r-1)} \right) \\
&- (r-1)\E^{\star 2} \star (\H-\E)^{\star (r-2)}\\
=& (2r) \mathsf{P}-(r-1) q^{-\E^\vee}\E,
\end{align*}
where in the second equality we have used Lemma \ref{lemma: cohomology classes in quantum cohomology}, in the third equality we have written $\H=\H-\E+\E$, and in the fourth equality we have used the first relation in \eqref{relations} and the fact that, by Lemma \ref{lemma: cohomology classes in quantum cohomology}, we have
$$
\mathsf{P}=\H \star (\H-\E)^{\star (r-1)} =(\H-\E)^{\star r}+ \E \star (\H-\E)^{\star (r-1)}.
$$
This concludes the proof.
\end{proof}

\subsection{Proof of Theorem \ref{virtual formulas for 1 pt blow-up}}
Finally, we can compute
\begin{align*}
\vTev^X_{g,n,d \H^\vee +k \E^\vee}&=\mathrm{Coeff}(\mathsf{P}^{\star n}\star \Delta^{\star g},\mathsf{P} q^{d \H^\vee +k \E^\vee}) \\
&=\sum_{m=0}^g (2r)^{g-m}(1-r)^m \binom{g}{m} \mathrm{Coeff}(\mathsf{P}^{\star n+g-m} \star \E^{\star m},\mathsf{P} q^{d \H^\vee+(k+m)\E^\vee}),
\end{align*}
and the proof of Theorem \ref{virtual formulas for 1 pt blow-up} is reduced to the following.
\begin{lemma}\label{QH_coeff}
For $m,k \in \Z_{\geq 0}$, $\ell,d \in \Z_{>0}$ with $\ell -d-m > 0$, $d \geq k$ and satisfying
\begin{equation}\label{eqn: equivalent of dim constraint}
(r+1)d-(r-1)k=r(\ell-1),
\end{equation}
we have
$$
\mathrm{Coeff}(\mathsf{P}^{\star \ell-m} \star \E^{\star m},\mathsf{P} q^{d \H^\vee+(k+m)\E^\vee})=\binom{\ell-d-m-1}{k}.
$$
\end{lemma}

\begin{remark}\label{rmk:QH_coeff}
Note that 
$$
\binom{\ell-d-m-1}{k}=0
$$
unless $\ell-d-m-1 \ge k$, which, using \eqref{eqn: equivalent of dim constraint}, is equivalent to $d \ge (2r-1)k$. In particular, we get $0$ unless $\ell \geq r+2$.
\end{remark}

\begin{proof}[Proof of Lemma \ref{QH_coeff} ]
We proceed by induction on $m$. Suppose that $m=0$. Then, conditions \eqref{dim constraint}, \eqref{stronger inequality d, ks}, and \eqref{eqn: range for d to have enumerativity} are all satisfied with $n=\ell$, so
$$
\mathrm{Coeff}(\mathsf{P}^{\star \ell},\mathsf{P}  q^{d \H^\vee+k\E^\vee})=\vTev^X_{0,\ell,d \H^\vee +k \E^\vee} =\binom{\ell-d-1}{k}
$$
by Theorems \ref{explicit formulas in genus 0} and \ref{SAE for blow-ups at 1 point}. 

Suppose that $m=1$. If $\ell > r+1$, then we can write
\begin{align*}
\mathsf{P}^{\star (\ell-1)} \star \E&=\mathsf{P}^{\star (\ell-2)} \star (\H-\E)^{\star (r-1)} q^{\H^\vee+\E^\vee} \\
&=\mathsf{P}^{\star (\ell-3)} \star \H \star(\E q^{-\E^\vee}) \star (\H-\E)^{\star (r-2)} q^{\H^\vee+\E^\vee} \\
&=\mathsf{P}^{\star (\ell-3)} \star (\H-\E)^{\star (r-2)} q^{2 \H^\vee+\E^\vee}\\
&\ \vdots \\
&= \mathsf{P}^{\star (\ell-r-1)}  q^{r \H^\vee+\E^\vee}.
\end{align*}
Taking the coefficient of $\mathsf{P}q^{d \H^\vee+(k+1)\E^\vee}$ on both sides and applying the $m=0$ case gives the claim. If instead $\ell < r+2$, then the same chain of equalities yields
$$
\mathsf{P}^{\star (\ell-1)} \star \E= (\H -\E)^{\star (r+1-\ell)}q^{(\ell-1)\H^\vee+\E^\vee}.
$$
In particular, the coefficient of $\mathsf{P}q^{d \H^\vee+(k+1)\E^\vee}$ is $0$. The claimed equality then follows from Remark \ref{rmk:QH_coeff}.

Finally, suppose that $m \geq 2$. Write 
\begin{align*}
\mathsf{P}^{\star (\ell-m)} \E^{\star m}&= \mathsf{P}^{\star (\ell-m-1)} \E^{\star (m-1)}(\H-\E)^{\star (r-1)} q^{\H^\vee +\E^\vee} \\
&= \mathsf{P}^{\star (\ell-m)} \E^{\star (m-2)} q^{\H^\vee+\E^\vee}- \mathsf{P}^{\star (\ell-m-1)}\E^{\star (m-1)}q^{\H^\vee}.
\end{align*}
Taking the coefficient of $\mathsf{P}q^{d \H^\vee+(k+1)\E^\vee}$ on both sides and using the inductive hypothesis completes the proof.
\end{proof}

This concludes the proof of Theorem \ref{virtual formulas for 1 pt blow-up}.

\begin{remark}
   The condition $n-d \geq 1$ in Theorem \ref{virtual formulas for 1 pt blow-up} (i.e., the condition $\ell -m -d \geq 1$ in Lemma \ref{QH_coeff} ) is necessary. For example, when $r=2$ and $n=g=d=k=1$, we have

    \begin{align*}
    \vTev^X_{g,n,d\H^\vee+k \E^\vee}&= \mathrm{Coeff}(\mathsf{P} \star \Delta, q^{d \H^\vee + k\E^\vee} \mathsf{P}) \\
    &=4\cdot \mathrm{Coeff}(q^{\H^\vee} \H, q^{ \H^\vee +\E^\vee} \mathsf{P}) - \mathrm{Coeff}(q^{\H^\vee +\E^\vee }  (\H-\E), q^{\H^\vee +2 \E^\vee} \mathsf{P}) \\
    &=0
    \end{align*}
    while the right-hand side of Theorem \ref{virtual formulas for 1 pt blow-up} is equal to $1$.
\end{remark}

\begin{remark}
    When $k=0$, in Lemma \ref{QH_coeff} we have 
    $$
    d=\frac{r (\ell-1)}{r+1} \in \mathbb{Z},
    $$
    and so it must be that $\ell>r+1$. In particular, as one would expect, the same proof gives 
    $$
    \mathsf{vTev}^X_{g,n,d \H^\vee}= (r+1)^g =\mathsf{vTev}^{\P^r}_{g,n,d \H^\vee}
    $$
    for all $g \ge 0$ and $n,d>1$ satisfying \eqref{dim constraint}.
\end{remark}

\bibliographystyle{plain} 
\bibliography{tev_blowup_final_arxiv}

\end{document}